\documentclass[11pt,reqno]{amsart}
\usepackage{geometry}
\usepackage{amssymb}
\usepackage{hyperref}
\usepackage{url}
\usepackage{booktabs}
\usepackage{amsfonts}
\usepackage{microtype} 
\usepackage{graphicx}
\usepackage{amsmath}
\usepackage{mathtools}
\usepackage{enumitem}
\usepackage{comment}
\usepackage{parskip}
\usepackage[capitalize]{cleveref}
\usepackage{tikz-cd}
\usepackage{tikz}
\usepackage{subcaption}
\usepackage{listings}
\usepackage{multirow}
\usepackage{soul}
\usepackage[hyperref,doi,url=false,style=alphabetic,maxbibnames=99]{biblatex}

\allowdisplaybreaks

\theoremstyle{plain}
\newtheorem{thm}{Theorem}
\newtheorem{lem}{Lemma}
\newtheorem{prop}{Proposition}

\newtheorem{cor}{Corollary}
\newtheorem{conj}{Conjecture}

\newtheorem{ques}{Question}

\newtheorem{subtheorem}{Theorem}
\theoremstyle{remark}
\newtheorem*{claim*}{Claim}

\theoremstyle{definition}

\newcommand{\Z}{\mathbb{Z}}

\newcommand{\R}{\mathbb{R}}
\newcommand{\C}{\mathbb{C}}

\newcommand{\eps}{\varepsilon}
\renewcommand{\phi}{\varphi}

\DeclareMathOperator{\supp}{supp}

\newcommand{\todo}[1]{{\hl{TODO: #1}}}

\newcommand{\sqnorm}[1]{\norm{#1}^2}
\newcommand{\inner}[2]{\langle {#1} , {#2}\rangle}

\DeclareMathOperator{\im}{Im}

\DeclareMathOperator{\id}{id}

\DeclareMathOperator{\rank}{rank}

\DeclareMathOperator{\length}{length}

\DeclareMathOperator{\Ker}{Ker}

\DeclareMathOperator{\GL}{GL}

\DeclareMathOperator{\PSL}{PSL}

\DeclareMathOperator{\tors}{tors}

\DeclareMathOperator{\Int}{Int}

\DeclareMathOperator{\inj}{inj}

\newcommand{\norm}[1]{\left\|#1\right\|}

\newcommand{\cont}{\textnormal{cont}}
\newcommand{\expnd}{\textnormal{exp}}

\begin{document}
\title[Hyperbolic 3-manifolds with spectral gap for coclosed 1-forms]{Hyperbolic 3-manifolds with uniform spectral gap for coclosed 1-forms}
\author[Abdurrahman]{Amina Abdurrahman}
\email{amina@ihes.fr}
\author[Adve]{Anshul Adve} 
\email{aadve@princeton.edu}
\author[Giri]{Vikram Giri}
\email{vikramaditya.giri@math.ethz.ch}
\author[Lowe]{Ben Lowe}
\email{loweb24@uchicago.edu}
\author[Zung]{Jonathan Zung}
\email{jzung@mit.edu}

\maketitle

{\centering \it Dedicated to Nicolas Bergeron, his pioneering work is an inspiration. \bigskip \par}

\begin{abstract}
    We study two quantifications of being a homology sphere for hyperbolic 3-manifolds, one geometric and one topological: the spectral gap for the Laplacian on coclosed 1-forms and the size of the first torsion homology group. We first construct a sequence of closed hyperbolic integer homology spheres with volume tending to infinity and a uniform coclosed 1-form spectral gap. This answers a question asked by Lin--Lipnowski. We also find sequences of hyperbolic rational homology spheres with the same properties that geometrically converge to a tame limit manifold. Moreover, we show that any such sequence must have unbounded torsion homology growth.
    Finally we show that a sequence of closed hyperbolic rational homology 3-spheres with uniformly bounded rank and a uniform coclosed 1-form spectral gap must have torsion homology that grows exponentially in volume.
\end{abstract}

\setcounter{tocdepth}{1}
\tableofcontents

\begin{section}{Introduction}

In this paper we investigate some links between spectral and topological information in sequences of hyperbolic $3$-manifolds.  Understanding the spectrum of the Laplacian and how it relates to the topology of a Riemannian manifold is a central problem in geometry.  We will study the spectrum of the Laplacian on differential forms for a closed hyperbolic 3-manifold.  By the Hodge theorem, the non-zero spectrum of the Laplacian on 1-forms is divided into eigenvalues corresponding to coexact 1-forms and eigenvalues that correspond to exact 1-forms, and the latter are exactly the non-zero eigenvalues for the Laplacian on functions. By Poincar\'e duality (via the Hodge star), the spectrum of the Laplacian on 2-forms is determined by the spectrum of the Laplacian on 1-forms.  We therefore restrict our attention in what follows to the coexact part of the spectrum of the Laplacian on 1-forms. 

Throughout this paper, we say that a non-negative, self-adjoint operator has a spectral gap if its smallest eigenvalue is non-zero; the spectral gap for such an operator is its smallest eigenvalue. In this convention, by the Hodge theorem, a closed 3-manifold has a spectral gap on coclosed 1-forms if and only if it is a rational homology sphere {(the coclosed 1-form eigenvalues are just the coexact 1-form eigenvalues together with zero with multiplicity $\dim H^1(M,\mathbb R)$)}. We view the existence of a spectral gap as a quantification of the topological property of having the same rational homology as a sphere.  

The size of the first non-zero eigenvalue of the Laplacian on coexact 1-forms is related to questions about the growth of torsion homology in sequences of hyperbolic 3-manifolds, as well as to questions about the vanishing of gauge-theoretic invariants associated to $M$.  In the latter connection, Lin--Lipnowski established a new relationship between the existence of solutions to the Seiberg--Witten equations on a hyperbolic 3-manifold and the geometry of that manifold \cite{ll22}. In particular, they showed that a hyperbolic rational homology $3$-sphere with a large enough spectral gap for the Laplacian acting on coexact 1-forms does not admit any irreducible solutions to the Seiberg--Witten equations, and verified that there are many hyperbolic rational homology spheres satisfying their condition. As a step towards constructing infinitely many examples to which their theorem applies, Lin--Lipnowski asked whether there exists an infinite sequence of closed hyperbolic rational homology 3-spheres with a uniform spectral gap on coexact 1-forms.  Our first theorem answers their question in the affirmative.

The study of the spectrum on coclosed 1-forms is particularly rich in the case of hyperbolic 3-manifolds. The following are some interesting sequences of hyperbolic 3-manifolds which do not have a uniform spectral gap on coclosed 1-forms. As observed by Bergeron--Clozel \cite{bergeron_clozel}, sequences of hyperbolic arithmetic congruence 3-manifolds can never have a uniform gap for their 1-form spectrum. Moreover, there are examples of hyperbolic rational homology 3-spheres constructed by Rudd \cite[Section 6]{rudd23} whose limit is not $\mathbb{H}^3$ but which still fail to have a uniform gap on coclosed 1-forms; in fact, their spectral gap is exponentially small in their volume.
Calegari--Dunfield \cite{Calegari2006} have constructed examples of hyperbolic rational homology 3-spheres with injectivity radius growing without bound, and which therefore (see Section~\ref{sec.injrad}) cannot have a uniform spectral gap.\\

In the second part of the paper we study the connection between spectral information attached to  a sequence of hyperbolic 3-manifolds and torsion homology growth.\\
This is motivated by work of Bergeron--\c{S}eng\"un--Venkatesh~\cite{bergeron_sengun_venkatesh} which studies the delicate interplay between the complexity of homology classes of arithmetic manifolds, the number of small eigenvalues of the Laplacian acting on $1$-forms, and the exponential growth of torsion homology in the volume in towers of coverings. Their work indicates a possible relation between a tower of arithmetic congruence hyperbolic $3$-manifolds having a uniform spectral gap for the $1$-form Laplacian (or less restrictively, few small eigenvalues) and this tower exhibiting exponential growth of torsion homology. B--\c{S}--V's work is based on a conjecture about the topological complexity of (surfaces representing homology classes in) arithmetic congruence manifolds, and crucially relies on the Cheeger--M\"uller theorem by linking torsion homology to regulators and analytic torsion.

One aim of our project is to understand in what way the behaviour predicted by B--\c{S}--V is specific to arithmetic congruence hyperbolic manifolds. Without using Cheeger--M\"uller's theorem, we show that any sequence of pointed closed hyperbolic rational homology $3$-spheres with volumes tending to infinity, a uniform spectral gap for coexact 1-forms, and whose fundamental groups can be generated by a uniformly bounded number of elements must have torsion homology that grows exponentially in volume.\\

We now give precise statements of our main results.   

	\begin{subtheorem}\label{thm:1}
		There exist infinitely many distinct hyperbolic integer homology 3-spheres with a uniform spectral gap on coexact 1-forms.
	\end{subtheorem}

A sequence of distinct hyperbolic manifolds with a uniform spectral gap necessarily has unbounded volume, because a uniform spectral gap implies a uniform lower bound on the injectivity radius (see Section~\ref{sec.injrad}). We modify the construction used in Theorem~\ref{thm:1} to provide examples of a sequence of rational homology 3-spheres that occur in a tower of covering maps. More precisely, we have
 \begin{subtheorem}\label{thm.12}
     There exists a tower of covers of a hyperbolic rational homology 3-sphere 
     $$\cdots \to M_3\to M_2 \to M_1 \to M_0$$ such that the $M_i$ have a uniform bound on $|H_1(M_i)|$ and a uniform spectral gap on coexact 1-forms.
 \end{subtheorem}
We give another modification of the construction in Theorem \ref{thm:1} to prove the following theorem:
\begin{subtheorem}\label{thm.13}
    There exist infinitely many distinct closed hyperbolic 3-manifolds $M_i$ for which the first non-zero eigenvalue of both the Laplacian on functions and the Laplacian on coexact 1-forms is uniformly bounded away from zero. 
\end{subtheorem}
We point out that the $M_i$ of Theorem \ref{thm.13} are not rational homology spheres, {so 0 is a 1-form eigenvalue (with high multiplicity)}.  Theorem~\ref{thm.13} can be seen as an attempt to produce high dimensional spectral expanders in the sense of~\cite{lubotzky_high_2017} in the hyperbolic 3-manifold setting (see Question~\ref{q.sar}). When the condition of hyperbolicity is relaxed, forthcoming work of the last author~\cite{jz} shows that one can indeed construct such 3-manifold expanders, which will consequently be rational homology spheres.

The sequence of integer homology spheres produced in the proof of Theorem~\ref{thm:1} does not have a tame limit. In the next theorem, we provide another construction that has a tame limit:

    \begin{subtheorem}\label{thm:11}
		There exists a sequence $M_i$ of pointed hyperbolic rational homology 3-spheres geometrically converging to a tame, infinite volume manifold $M$ with a uniform spectral gap on coexact 1-forms.
	\end{subtheorem}

The manifolds in the sequence are rational homology 3-spheres 
 with a uniformly bounded number of generators for their fundamental groups. 
 The size of the torsion homology groups $H_1(M_i;\mathbb{Z})_{\text{tors}}$ is unbounded along this sequence. The next theorem shows that this is always the case for sequences of hyperbolic 3-manifolds with tame limits and a uniform spectral gap for coexact 1-forms.  
 
 
	\begin{thm}\label{thm:2}
            Suppose $M_i$ is a sequence of closed hyperbolic 3-manifolds geometrically converging to a tame manifold $M$ with at least one end. If there is a uniform upper bound on $|H_1(M_i;\Z)|$, then the spectral gap for coexact 1-forms goes to zero.
	\end{thm}
The estimates used to prove \cref{thm:2} are effective and can be adapted to prove the following corollary.  
\begin{cor}\label{cor:21}
    There are only finitely many closed hyperbolic 3-manifolds with spectral gap for coexact 1-forms bounded below, $|H_1(M;\Z)|$ bounded above, and whose fundamental groups can be generated by a given finite number of elements.
\end{cor}

In fact, we are able to prove something stronger.  

\begin{cor}\label{cor:22}
    Suppose $M_i$ is a sequence of hyperbolic rational homology 3-spheres with a uniform bound on the number of generators of their fundamental groups $\pi_1(M_i)$ and a uniform coexact 1-form spectral gap.  Then $|H_1(M_i;\mathbb{Z})_{\textnormal{tors}}|$ grows exponentially in volume.
\end{cor}

Corollary~\ref{cor:21} is an analogue of the result proved in \cite{bs11} dealing with the spectral gap for the standard Laplacian.

\subsection{Relation to other work} 
    The sequence of hyperbolic 3-manifolds in Theorem~\ref{thm:11} is obtained by gluing two homology handlebodies $E^\pm$ with incompressible boundaries $\Sigma$ by powers of a fixed pseudo-Anosov map $\varphi$. Similar constructions have also been used in works of Brock--Dunfield and Rudd (see~\cite{bd17, rudd23}) with very different conclusions. In~\cite[Section 7]{bd17}, this construction is used in the proof of their Theorem 1.5 which shows the existence of rational homology spheres $M_i$ with a uniform injectivity radius lower bound so that the Thurston norm of generators of $H^1(M_i; \mathbb{Z})$ grows exponentially in volume. In~\cite[Section 6]{rudd23}, it is shown that there exists a sequence of rational homology spheres with coexact 1-form spectral gap that is {\emph{exponentially small}} in their volume! The difference in constructions lies in the choice of pseudo-Anosov map:~\cite{rudd23} chooses the map so that the subspace of $H_1(\Sigma;\mathbb{R})$ that is killed in the left handlebody intersects the expanding subspace of the action of $\varphi$ on $H_1(\Sigma;\mathbb{R})$. This condition is not generic and, indeed, in our case we need this intersection to be empty in order to obtain a uniform spectral gap.

The Cheeger inequality bounds from below the first Laplace eigenvalue of a closed Riemannian manifold $M$ in terms of the  Cheeger constant, which informally speaking measures how large a hypersurface that cuts $M$ into two pieces of equal volume must be.  Yau \cite[Problem 79]{yau82} asked whether it is possible to bound the bottom eigenvalue of the spectrum of the Laplacian on differential forms, given bounds on the geometry of $M$ (for example, bounds on its curvature, diameter, or injectivity radius.)     

Recent work by Lipnowski--Stern \cite{ls18} and Boulanger--Courtois \cite{bc22cheeger} have given very satisfactory general answers to this question for the coexact 1-form spectrum in the case of respectively closed hyperbolic 3-manifolds and general Riemannian manifolds.  The analogue here of the Cheeger constant for the coexact 1-form spectrum is the stable isoperimetric constant for 1-dimensional cycles, which measures the asymptotic geometric cost of bounding multiples of null-homologous cycles.  See Section \ref{sec:isoperimetric} for a more detailed description of their results and why, although quantities similar to the stable isoperimetric constant are also important in this paper, we cannot directly apply their work here.

\begin{subsection}{Bass note spectra}\label{sec:bassnote}
     It is natural to reformulate our results in terms of various bass note spectra of hyperbolic 3-manifolds. For a family $\mathcal{F}$ of Hilbert spaces $\mathcal{H}_i$ with a choice of nonnegative self-adjoint operator $\Delta_i : \mathcal{H}_i\to \mathcal{H}_i$, we define the \emph{bass note spectrum} of $\mathcal{F}$ to be
     \begin{equation*}
         \text{Bass}_{\mathcal{F}} := \overline{\{\text{smallest eigenvalue of }\Delta_i : (\mathcal{H}_i,\Delta_i)\in \mathcal{F}\}}.
     \end{equation*}
    The bass note spectrum has been studied in the case of function spectra (usually restricting to mean zero functions), e.g.
    \begin{align*}
        \mathcal{F}=\{(L^2(M_i)^{\text{mean-zero}},\Delta):M_i \text{ a hyperbolic/arithmetic/congruence surface}\},
    \end{align*}
    see~\cite{kmp},~\cite{magee2024limit}; in the case of 3-regular graphs, see~\cite{ks21},~\cite{alon_wei}. {The lectures of Sarnak~\cite{sarnak_chern} give a nice overview of the concept and various results associated to bass note spectra.} The major open Selberg $1/4$ conjecture can be strengthened using this language as
    \begin{conj}[Sarnak's Selberg+epsilon Conjecture]
        The bass note spectrum of the Laplacian on congruence surfaces is an infinite set of numbers $\geq 1/4$ whose only limit point is $1/4$.
    \end{conj}
    We say a bass note spectrum is \emph{rigid} if it has at most one limit point (usually this limit point is $0$). Otherwise, it is \emph{non-rigid}. 
    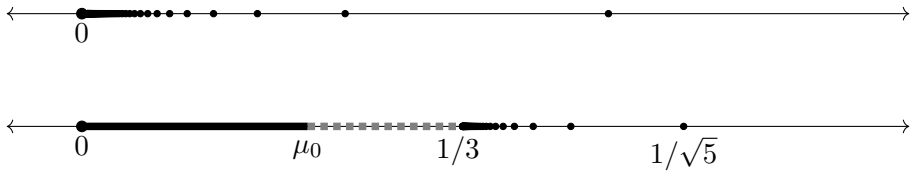
\begin{figure}[ht]
    \centering
    \begin{tikzpicture}
        \draw[<->] (-6,0) -- (6,0);
        \draw[<->] (-6,1.5) -- (6,1.5);
        \foreach \n in {1,...,100}
        {
        \filldraw[black] (7/\n-5,1.5) circle (1.2pt);
        }
        \foreach \n in {1,...,50}
        {
        \filldraw[black] (3/\n,0) circle (1.2pt);
        }
        \draw[line width=1mm , dotted, gray] (-2,0) -- (0,0);
        \draw[line width=1mm , black] (-5,0) -- (-2,0);
        \filldraw[black] (-5,1.5) circle (2pt) node[anchor=north]{0};
        \filldraw[black] (-5,0) circle (2pt) node[anchor=north]{0};
        \filldraw[black] (3,0) circle (0pt) node[anchor=north]{$1/\sqrt{5}$};
        \filldraw[black] (0,0) circle (0pt) node[anchor=north]{$1/3$};
        \filldraw[black] (-2,-0.05) circle (0pt) node[anchor=north]{$\mu_0$};
        \end{tikzpicture}
    \caption{These sketches illustrate the rigidity/non-rigidity behaviour of bass-note spectra.  Top: A rigid spectrum which is infinite and discrete with the unique accumulation point $0$.
    Bottom: A non-rigid spectrum; here showing the Markoff spectrum (see \cite{series} and Chapter 7 of~\cite{cf89} for nice expositions) having a continuous (``flexible'') bottom part, a discrete (``rigid'') top part, and a fractal transition part.}
    \label{fig}
\end{figure}

    Our main theorems can now be stated in terms of the bass note spectra $\text{Bass}_{\mathcal{F}(\mathcal{M})}$, where 
    $$ \mathcal{F}(\mathcal{M}) = \{(L^2(M,T^*M)^{\text{coexact}}, \Delta_1): M\in \mathcal{M}\},$$
    $\mathcal{M}$ is a family of hyperbolic 3-manifolds, and $\Delta_1$ is the 1-form Laplacian. We consider the following cases:
    \begin{enumerate}
        \item $\mathcal{M}_1$ consists of hyperbolic manifolds which are integer homology 3-spheres.
        \item $\mathcal{M}_2$ consists of hyperbolic 3-manifolds with homology groups of cardinality at most $N$.
        \item $\mathcal{M}_3$ consists of hyperbolic rational homology 3-spheres with fundamental group having at most $L$ generators.
    \end{enumerate}
    Then, as a consequence of our theorems, we have
     \begin{thm}
         $\textnormal{Bass}_{\mathcal{F}(\mathcal{M}_1)}$ is non-rigid, and for any $L,N$ large, $\textnormal{Bass}_{\mathcal{F}(\mathcal{M}_2)}$ and $\textnormal{Bass}_{\mathcal{F}(\mathcal{M}_3)}$ are non-rigid. On the other hand, $\textnormal{Bass}_{\mathcal{F}(\mathcal{M}_2 \cap \mathcal{M}_3)}$ is rigid.
     \end{thm}
 \end{subsection}

 \begin{subsection}{Organization}
     The paper is organized as follows: In Section~\ref{sec.heu} we provide heuristics for our constructions, describe the relation of our paper to work by Lott ~\cite{lott}, and explain some of the elementary linear algebra that underlies our proofs. In Section~\ref{sec.prelim}, we  explain some structural results on hyperbolic 3-manifolds that will be important for our constructions, in particular the work by Brock--Minsky--Namazi--Souto \cite{bmns16} on effective geometrization. Their work provides strong control on the coarse geometry of hyperbolic manifolds constructed from a finite set of building blocks. In Section~\ref{sec:construction}, we describe the construction of the manifolds that demonstrate Theorem~\ref{thm:1} and conclude with a sketch providing intuition for the proofs of Theorems~\ref{thm:1} and \ref{thm:11}. In Section~\ref{sec.var}, we describe how the construction described in the previous section can be modified to give sequences that demonstrate Theorems~\ref{thm.12} and \ref{thm.13}. In Section~\ref{sec.bddprim}, we prove a couple of technical propositions that construct primitives for exact forms with an $L^2$ bound. In Sections~\ref{sec.pf:1}, \ref{sec.pf:11}, and \ref{sec.pf:2}, we prove Theorems~\ref{thm:1}, \ref{thm:11}, and \ref{thm:2} respectively. In the final Section~\ref{sec.q}, we state some open questions. 
 \end{subsection}

\end{section}

\begin{subsection}*{Acknowledgements}
We wish to thank Peter Sarnak for his enthusiasm and for many insightful questions and discussions out of which this project evolved. We are also grateful for interest in this paper and feedback from many others, including I. Biringer, N. Dunfield, M. Fraczyk, N. Kravitz, F. Lin, M. Lipnowski, J. Lott, W. L\"uck, J. Raimbault, C. Rudd, J. Souto, and M. Stern.

\end{subsection}

\begin{section}{Heuristics}\label{sec.heu}
\begin{subsection}{Hodge decomposition}
    On a Riemannian 3-manifold, the Hodge theorem gives the splitting of the de Rham complex shown below. Here $X_i$ is the space of coexact $i$-forms, $Z_i$ is the space of exact $i$-forms, and $\mathcal H_i$ is the space of harmonic $i$-forms.
\begin{equation}
\begin{tikzcd}
X_0 \arrow[d, "\oplus", phantom] \arrow[rdd, "d_0", shift left] & X_1 \arrow[d, "\oplus", phantom] \arrow[rdd, "d_1", shift left] & X_2 \arrow[d, "\oplus", phantom] \arrow[rdd, "d_2", shift left] & 0 \arrow[d, "\oplus", phantom]            \\
\mathcal{H}_0 \arrow[d, "\oplus", phantom]                      & \mathcal H_1 \arrow[d, "\oplus", phantom]                       & \mathcal H_2 \arrow[d, "\oplus", phantom]                       & \mathcal H_3 \arrow[d, "\oplus", phantom] \\
0                                                               & Z_1 \arrow[luu, "d_0^*", shift left]                            & Z_2 \arrow[luu, "d_1^*", shift left]                            & Z_3 \arrow[luu, "d_2^*", shift left]     
\end{tikzcd}
\end{equation}
The nonzero spectrum of the Laplacian is divided into eigenvalues corresponding to the singular values of $d_0$, $d_1$, and $d_2$. By Poincar\'e duality, the singular values of $d_0$ coincide with the singular values of $d_2$. These are the square roots of the eigenvalues of the usual Laplacian on functions. The remaining spectrum is determined by the singular values of $d_1$, i.e., the square roots of the eigenvalues of the Laplacian acting on coexact 1-forms. The spectrum for coclosed 1-forms is the same as the spectrum for coexact 1-forms except for the addition of zero with multiplicity $\dim(\mathcal H_1)$.
\end{subsection}
\begin{subsection}{Constructions}
We present two main examples of sequences of manifolds having a spectral gap for coexact 1-forms. The first is composed of a sequence of blocks $B_i$ glued together end to end. See panel 1 of \cref{fig:filling2}. The building block $B_i$ is chosen so that the map $H_1(B_i,\R) \to H_1(B_{i-1}\cup B_i \cup B_{i+1},\R)$ is the zero map.

Our second construction is given by gluing two homology handlebodies $E^-$ and $E^+$ by a large power of a mapping class. We choose the gluing so the resulting manifold is a rational homology sphere. A schematic picture of the resulting manifold is shown in panel 1 of \cref{fig:filling}. There is a long region in the middle homeomorphic to $\Sigma \times [-n,n]$.

\end{subsection}

\begin{subsection}{Stable isoperimetric ratio}\label{sec:isoperimetric}
We define the \emph{stable isoperimetric ratio} of a Riemannian 3-manifold $M$ as
$$\sup_{\gamma} \inf_{\partial \Sigma = \gamma^n} \frac{\text{Area}(\Sigma)}{n\, \text{length}(\gamma)}.$$ In this formula, $\gamma$ runs over all rationally null-homologous closed curves in $M$, and $\Sigma$ runs over all oriented surfaces which span some multiple cover of $\gamma$. Note that some authors refer to the reciprocal of this quantity as the isoperimetric ratio. The Cheeger-type theorems of Boulanger--Courtois and Lipnowski--Stern show that a gap in the spectrum for coexact 1-forms is related to an upper bound on the stable isoperimetric ratio. Before showing that our manifolds have a uniform spectral gap, let us explain why they have small stable isoperimetric ratios. Note that the constants in the comparison inequalities of Boulanger--Courtois and Lipnowski--Stern get worse as the volume or diameter of the manifold increases, so one cannot directly use their results to show that a growing sequence of manifolds has a uniform spectral gap.

Let $M$ be the first example we described, a manifold built out of a sequence of blocks $B_i$. Suppose $\gamma$ is a curve in $M$ as shown in the first panel of \cref{fig:filling2}. First, cut up $\gamma$ into several smaller closed curves, each of which is contained in one of the blocks $B_i$. With our choice of $B_i$, any closed curve in $B_i$ is null-homologous in $B_{i-1} \cup B_i \cup B_{i+1}$. So we may find a surface spanning each of the remaining curves. These surfaces each have controlled area because they live in a region spanning at most three blocks.

\begin{figure}[ht]
    \centering
\begingroup%
  \makeatletter%
  \providecommand\color[2][]{%
    \errmessage{(Inkscape) Color is used for the text in Inkscape, but the package 'color.sty' is not loaded}%
    \renewcommand\color[2][]{}%
  }%
  \providecommand\transparent[1]{%
    \errmessage{(Inkscape) Transparency is used (non-zero) for the text in Inkscape, but the package 'transparent.sty' is not loaded}%
    \renewcommand\transparent[1]{}%
  }%
  \providecommand\rotatebox[2]{#2}%
  \newcommand*\fsize{\dimexpr\f@size pt\relax}%
  \newcommand*\lineheight[1]{\fontsize{\fsize}{#1\fsize}\selectfont}%
  \ifx\svgwidth\undefined%
    \setlength{\unitlength}{323.64961399bp}%
    \ifx\svgscale\undefined%
      \relax%
    \else%
      \setlength{\unitlength}{\unitlength * \real{\svgscale}}%
    \fi%
  \else%
    \setlength{\unitlength}{\svgwidth}%
  \fi%
  \global\let\svgwidth\undefined%
  \global\let\svgscale\undefined%
  \makeatother%
  \begin{picture}(1,0.68469839)%
    \lineheight{1}%
    \setlength\tabcolsep{0pt}%
    \put(0,0){\includegraphics[width=\unitlength,page=1]{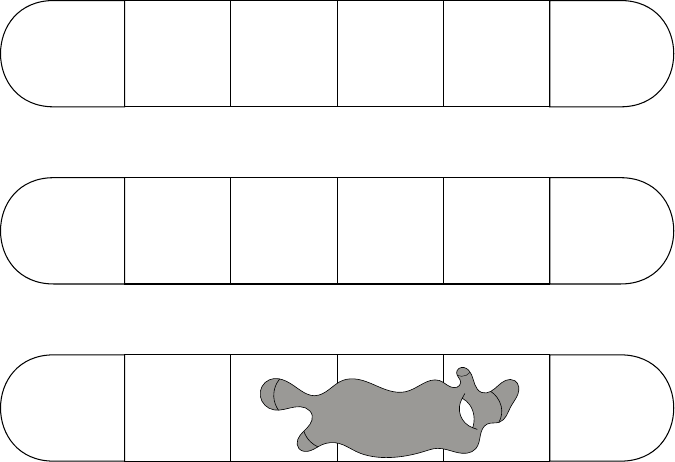}}%
    \put(0.10235835,0.48478289){\makebox(0,0)[t]{\lineheight{1.25}\smash{\begin{tabular}[t]{c}$B_0$\end{tabular}}}}%
    \put(0.25907266,0.48541024){\makebox(0,0)[t]{\lineheight{1.25}\smash{\begin{tabular}[t]{c}$B_1$\end{tabular}}}}%
    \put(0.73756371,0.483417){\makebox(0,0)[t]{\lineheight{1.25}\smash{\begin{tabular}[t]{c}$B_{n-1}$\end{tabular}}}}%
    \put(0.90767659,0.48446713){\makebox(0,0)[t]{\lineheight{1.25}\smash{\begin{tabular}[t]{c}$B_n$\end{tabular}}}}%
    \put(0,0){\includegraphics[width=\unitlength,page=2]{filling2.pdf}}%
  \end{picture}%
\endgroup%

    \caption{}
    \label{fig:filling2}
\end{figure}

Now let $M$ be the second example we described, sketched in \cref{fig:filling}. As in the previous example, we can cut up any closed curve $\alpha$ into smaller pieces, each of which is contained in $\Sigma \times [i-1,i+1]$ for some $i$. Since $\alpha$ is null-homologous, it can be written as a rational linear combination of two 1-cycles $\alpha^-$ and $\alpha^+$ such that $\alpha^+$ is null-homologous in $E^+$ and $\alpha^-$ is null-homologous in $E^-$. It appears as though a surface spanning $\alpha^-$ will have to have large area since it must reach all the way into $E^-$. 
However, for a generic mapping class, the length of the shortest representative of the rational homology class of $\alpha^-$ will shrink exponentially fast as we push it towards $E^-$. So we can construct a rational 2-chain of small area spanning $\alpha^-$ by a ``push and simplify'' procedure. Push $\alpha^-$ one block to the left, simplify $\alpha^-$ to the shortest representative in its homology class in $\Sigma\times (i,i+1)$, and repeat. Once $\alpha^-$ arrives at $E^-$, cap it off. By the exponential decay property, this procedure sweeps out a rational 2-chain spanning $\alpha^-$ whose area is proportional to $\length(\alpha^-)$. The same holds for $\alpha^+$. 

\begin{figure}[ht]
\centering
\begingroup%
  \makeatletter%
  \providecommand\color[2][]{%
    \errmessage{(Inkscape) Color is used for the text in Inkscape, but the package 'color.sty' is not loaded}%
    \renewcommand\color[2][]{}%
  }%
  \providecommand\transparent[1]{%
    \errmessage{(Inkscape) Transparency is used (non-zero) for the text in Inkscape, but the package 'transparent.sty' is not loaded}%
    \renewcommand\transparent[1]{}%
  }%
  \providecommand\rotatebox[2]{#2}%
  \newcommand*\fsize{\dimexpr\f@size pt\relax}%
  \newcommand*\lineheight[1]{\fontsize{\fsize}{#1\fsize}\selectfont}%
  \ifx\svgwidth\undefined%
    \setlength{\unitlength}{323.64961399bp}%
    \ifx\svgscale\undefined%
      \relax%
    \else%
      \setlength{\unitlength}{\unitlength * \real{\svgscale}}%
    \fi%
  \else%
    \setlength{\unitlength}{\svgwidth}%
  \fi%
  \global\let\svgwidth\undefined%
  \global\let\svgscale\undefined%
  \makeatother%
  \begin{picture}(1,0.68469839)%
    \lineheight{1}%
    \setlength\tabcolsep{0pt}%
    \put(0,0){\includegraphics[width=\unitlength,page=1]{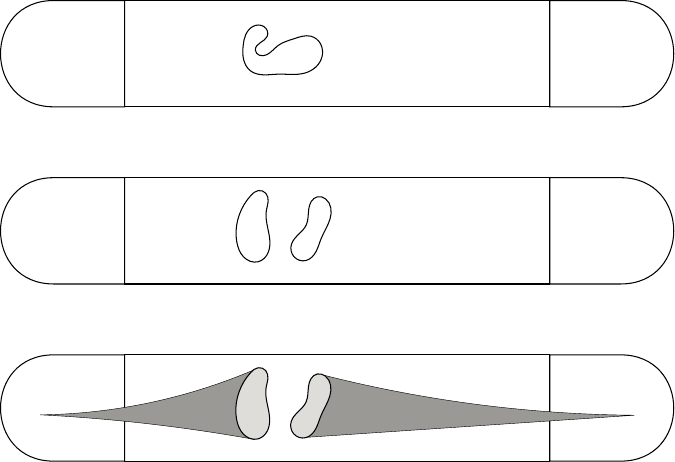}}%
    \put(0.10815544,0.48980724){\makebox(0,0)[t]{\lineheight{1.25}\smash{\begin{tabular}[t]{c}$E^-$\end{tabular}}}}%
    \put(0.34517565,0.3326036){\makebox(0,0)[rt]{\lineheight{1.25}\smash{\begin{tabular}[t]{r}$\alpha^-$\end{tabular}}}}%
    \put(0.4949864,0.33109237){\makebox(0,0)[lt]{\lineheight{1.25}\smash{\begin{tabular}[t]{l}$\alpha^+$\end{tabular}}}}%
    \put(0.49559307,0.48658459){\makebox(0,0)[t]{\lineheight{1.25}\smash{\begin{tabular}[t]{c}$\Sigma\times[-n,n]$\end{tabular}}}}%
    \put(0.89478205,0.49044411){\makebox(0,0)[t]{\lineheight{1.25}\smash{\begin{tabular}[t]{c}$E^+$\end{tabular}}}}%
    \put(0.49046388,0.59359888){\makebox(0,0)[lt]{\lineheight{1.25}\smash{\begin{tabular}[t]{l}$\alpha$\end{tabular}}}}%
  \end{picture}%
\endgroup%

\caption{}
\label{fig:filling}
\end{figure}

The proof that our manifolds have spectral gap will mimic the arguments just given, but $\gamma$ will be replaced by an exact 2-form.

\end{subsection}

\subsection{Some linear algebra}
Our analysis of the stable isoperimetric inequality above is predicated on exponential decay of certain homology classes in $H_1(\Sigma,\R)$. When can we guarantee that the lengths of the homology classes $\alpha^+$ and $\alpha^-$ decay exponentially quickly as we push them towards $E^-$ and $E^+$?

Let $\phi^{2N}$ be the mapping class used to glue $E^-$ and $E^+$. Let $A$ be the action of $\phi$ on $H_1(\Sigma,\R)$. {Assume that $A$ is a hyperbolic matrix, meaning that it has a $g$-dimensional contracting subspace spanned by $g$ eigenvectors with eigenvalues of norm less than 1, and a complementary $g$-dimensional expanding subspace spanned by $g$ eigenvectors with eigenvalues of norm greater than one.} Suppose $\Lambda^+$ and $\Lambda^-$ are two rank $g$ sublattices of $\Z^{2g}$ corresponding to the subspaces of $H_1(\Sigma,\Z)$ which are killed in $E^+$ and $E^-$ respectively. We can arrange that $\alpha^+ \in \Lambda^+\otimes \R$ and $\alpha^-\in \Lambda^-\otimes \R$. 

The condition of exponential decay is precisely that $\Lambda^+$ has zero intersection with the contracting subspace of $A$, and $\Lambda^-$ has zero intersection with the expanding subspace of $A$. The following lemma explains the connection between this condition and the size of the torsion homology of $M$. It is the basic piece of linear algebra underlying the proof of \cref{thm:2} and Corollary \ref{cor:22}.

\begin{lem}
    Let $A \in \GL_{2g}(\mathbb{Z})$ be a hyperbolic matrix. Let $\Lambda^{\pm} \subseteq \mathbb{Z}^{2g}$ be rank $g$ sublattices such that $\Lambda^+$ (resp. $\Lambda^-$) has zero intersection with the contracting (resp. expanding) subspace. Denote $\Lambda_N^{\pm} = A^{\pm N}\Lambda^{\pm}$. Then
    \begin{align*}
        \#\Big(\frac{\mathbb{Z}^{2g}}{\Lambda_N^+ + \Lambda_N^-}\Big)
    \end{align*}
    grows exponentially with $N$.
\end{lem}

\begin{proof}
    Let $v \in \Lambda_N^+ + \Lambda_N^-$ be nonzero. It suffices to show that $|v|$ is exponentially large. Write $v = v_N^+ + v_N^-$ with $v_N^{\pm} \in \Lambda_N^{\pm}$. Denote $V_N^{\pm} = \Lambda_N^{\pm} \otimes \mathbb{R}$. Then $V_N^+$ converges to the expanding subspace, and $V_N^-$ converges to the contracting subspace, so $V_N^+$ and $V_N^-$ are transverse uniformly in $N$. Therefore $|v| \gtrsim |v_N^{\pm}|$.

    One of $v_N^{\pm}$ is nonzero, say $v_N^+$. Write $v_N^+ = A^N u^+$ with $u^+ \in \Lambda^+$ (note $u^+$ depends on $N$). Since $u^+$ is integral, it has norm $\geq 1$. In addition, $u^+$ is transverse to the contracting subspace uniformly in $N$. It follows that $|A^N u^+|$ grows exponentially in $N$. Now
    \begin{align*}
        |v|
        \gtrsim |v_N^+|
        = |A^N u^+|.
    \end{align*}
    We conclude that $|v|$ is exponentially large, as desired.
\end{proof}

\begin{subsection}{Random methods}
In the setting of surfaces or graphs, a spectral gap for functions is a \emph{generic} property. By this we mean that for many random models, a random graph or surface is an expander with high probability. We don't know a model for random 3-manifolds which produces manifolds having a uniform spectral gap for coexact 1-forms. 

The Dunfield--Thurston model produces a Heegaard splitting by gluing together two handlebodies by a random word in the mapping class group of a surface \cite{dunfield_finite_2006}. A random word of length $N$ has subsequences of length $\sim \log(N)$ which act trivially on  homology, and the techniques of \cref{thm:torsion_analytic} show that the resulting hyperbolic 3-manifold has spectral gap for coexact 1-forms going to zero.

Petri and Raimbault suggested another model of random hyperbolic 3-manifolds: they randomly glue together $N$ truncated tetrahedra, and then double the resulting manifold with boundary \cite{petri_model_2022}. While these manifolds likely have spectral gap for functions, they fail to have uniform spectral gap for coexact 1-forms because their injectivity radii tend to zero. Moreover, these manifolds are never rational homology spheres.

The failure of random methods to produce examples of higher expanders is well known in the context of simplicial complexes. See Section 4 of \cite{lubotzky_high_2017} for an overview.
\end{subsection}
\begin{subsection}{Injectivity radius}\label{sec.injrad}
An obstruction to obtaining a uniform spectral gap for coexact 1-forms is that the manifolds $M_i$ we construct must have uniform 2-sided bounds on the injectivity radius at any point. Indeed, if the injectivity radius goes to infinity or to zero along a sequence of points $p_i\in M_i$, then as discussed in \cite[Proposition 5.1]{ll22}, one may construct explicit coexact 1-forms on $M_i$ with Rayleigh quotient tending to zero. Thus any such sequence, e.g. congruence towers or more generally sequences which Benjamini--Schramm converge to $\mathbb{H}^3$, are doomed to fail. This obstruction is not present in the case of the function spectrum of hyperbolic surfaces or graphs where the random objects do indeed Benjamini--Schramm converge to their corresponding symmetric space.
    
\end{subsection}

\subsection{Relation to~\texorpdfstring{\cite{lott}}{lott}} We expect that one could prove \cref{thm:11} using Lott's computation of the $L^2$ cohomology of tame hyperbolic 3-manifolds with degenerate ends. Suppose a sequence of pointed homology $3$-spheres $(M_k,p_k)$ converges to an infinite volume limit $(M,p)$ with a single degenerate end $\Sigma \times [0,\infty)$. If the reduced $L^2$-cohomology of $M$ is non-trivial (that is, $\overline{H}^1_{(2)}(M) := \Ker(d)/\overline{\im(d)} \neq \{0\}$), then, by a cut-off procedure, we should see that the coexact $1$-form spectral gap of the $M_k$ must tend to $0$. Contrapositively, if the ends of a ``cigar" construction had vanishing $L^2$-cohomology and looked like Lott's doubly-degenerate example in the `middle' (see Section 4 of~\cite{lott}), then the $M_k$ would have a uniform coexact $1$-form gap.  

To compute the reduced $L^2$-cohomology in our case, we use the exact sequence in Proposition 21 of~\cite{lott}:
\begin{align*}
    0 \to \im(H^1(K,\partial K;\mathbb{R}) \to H^1(K;\mathbb{R})) \to \overline{H}^1_{(2)}(M) \to L_1\cap L_2 \to 0
\end{align*}
where $L_1$ and $L_2$ are two Lagrangian subspaces of $H^1(\Sigma;\mathbb{R})$ defined as
\begin{align*}
    L_1 &:= \im(H^1(K;\mathbb{R})\to H^1(\Sigma;\mathbb{R}))\,,\\
    L_2 &:= \{h\in H^1(\Sigma;\mathbb{R}) : \int_0^\infty \langle h,h \rangle_{t} \,dt <\infty\},
\end{align*}
where the bilinear form is as in equation (6.6) of \cite{lott}. This definition of $L_2$ follows from the remark just before Proposition 16 in \cite{lott}.

In our case, $K$ is an end cap of the homology spheres and $\partial K = \Sigma$, and so we have $\im(H^1(K,\partial K;\mathbb{R}) \to H^1(K;\mathbb{R})) = \{0\}$. Thus we get that $\overline{H}^1_{(2)}(M) \simeq L_1\cap L_2$. From the definition of $L_2$ and equation (7.3) of \cite{lott}, $L_2$ is the contracting subspace for the action of $\phi^*$ on $H^1(\Sigma;\mathbb{R})$ in the case of a mapping cylinder. We note that $L_1\cap L_2 = \emptyset$ is one of the conditions we need (condition (v) of Section~\ref{sec.pf:11}) on a sequence of manifolds for it to have a uniform gap. 

\end{section}
\begin{section}{Preliminaries on hyperbolic 3-manifolds}\label{sec.prelim}

\subsection{Kleinian Groups}
A Kleinian group $\Gamma$ is a discrete subgroup of $\PSL_2(\mathbb C)$, the group of orientation preserving isometries of $\mathbb{H}^3$. Any complete hyperbolic 3-manifold $M$ can be written as	$  \mathbb{H}^3 / \Gamma $ for some Kleinian group $\Gamma$.  The Kleinian group $\Gamma$ can be viewed as the image of a discrete and faithful representation from $\pi_1(M)$ to $\PSL_2(\mathbb{C})$.

	
	The sequence $\Gamma_i$ converges \textit{geometrically} to $\Gamma$ if there are choices of basepoints $p_i$ for $\mathbb{H}^3 / \Gamma_i$ and $p$ for $\mathbb{H}^3 / \Gamma$, and a sequence of balls $B_i \subset \mathbb{H}^3$ that exhausts $\mathbb{H}^3$, such that the center of $B_i$ projects to the basepoint $p_i$, and each $B_i/\Gamma_i$ can be mapped $k_i$-quasi-isometrically onto a subspace of $\mathbb{H}^3/\Gamma$ by differentiable maps $F_i$, with  $k_i \rightarrow 1$ as $i\rightarrow \infty$ and $F_i(p_i)=p$ for all $i$.   We point out that geometric convergence of $\Gamma_i$ implies  pointed Gromov--Hausdorff convergence of $\Gamma_i \backslash \mathbb{H}$ for some choice of basepoints (see \cite[Section 3]{bs11}.) In what follows the sequences of hyperbolic 3-manifolds that we consider will have a uniform lower bound for their injectivity radii, and the following proposition will be useful. 
 
 \begin{prop}[{\cite[Corollary 3.3]{bs11}}] \label{geometriccompactness}
Fix $\epsilon>0$. Then every sequence of closed, pointed hyperbolic 3-manifolds $(M,p)$ with injectivity radius at $p$ greater than $\epsilon$  has a subsequence that converges geometrically.   
 \end{prop}

 A 3-manifold is \textit{tame} if it is homeomorphic to the interior of a compact 3-manifold (possibly with boundary.)  The \textit{rank} of a manifold is the minimal number of generators for its fundamental group. Agol and Calegari--Gabai proved Marden's conjecture that a hyperbolic 3-manifold with bounded rank is tame \cite{agol2004tameness}, \cite{cgtameness}.  Although we will not directly use this fact, we comment that a geometric limit of a sequence of pointed hyperbolic 3-manifolds with a lower bound for their injectivity radius and uniformly bounded rank is tame.  This can be shown using Proposition \ref{geometriccompactness} and \cite[Theorem 14.4]{bs23}.

\begin{subsection}{Gluing theorems for hyperbolic 3-manifolds}\label{sec.gluing}

Effective control on the geometry of hyperbolic 3-manifolds built out of finitely many pieces glued together will be important in what follows. This is provided by work by Brock--Minsky--Namazi--Souto \cite{bmns16}, which we summarize now.

Let $\Sigma$ be a compact surface, possibly with boundary.   The \textit{curve graph} of $\Sigma$ is a graph whose vertices are non-peripheral and homotopically non-trivial isotopy classes of essential loops in $\Sigma$, and whose edges join pairs of isotopy classes that have disjoint representatives.  The curve graph of $\Sigma$ is connected, and the distance  between two vertices in the curve graph is simply the smallest number of edges on a path joining them.  

A \textit{marking} $\mu$ on a surface $\Sigma$ is a set of isotopy classes of essential simple loops on $\Sigma$, along with a choice of at most one \textit{transversal} for each element of $\mu$ {which is disjoint from any other element in $\mu$}.  A \textit{transversal} for a loop $\alpha \in \mu$ is an essential simple loop $\beta$ which either intersects $\alpha$ once, or twice and the regular neighborhood of $\alpha \cup \beta$ is homeomorphic to a 4-holed sphere.  A \textit{complete marking} is one that is maximal, in the sense that the curves in the marking give a pants decomposition of $\Sigma$, and there is a transveral for every element of the marking.  

A \textit{decorated 3-manifold} is a compact  oriented 3-manifold $M$ whose boundary components each have genus at least 2  and are equipped with a complete marking. For a finite collection $\mathcal{M}$ of decorated 3-manifolds, an $\mathcal{M}$\textit{-gluing} is a 3-manifold $X$ obtained by gluing boundary components of copies of elements of $\mathcal{M}$.  Each boundary pairing of marked boundary components $(\Sigma_1,\mu_1)$ and $(\Sigma_2,\mu_2)$  is specified by an orientation reversing homeomorphism between $\Sigma_1$ and $\Sigma_2$.  

Given a boundary pairing $\Phi:\Sigma_1\to \Sigma_2$, we define the marking $\Phi_*(\mu_1)$ to be the marking on $\Sigma_2$ consisting of the images of the loops in $\mu_1$ under $\Phi$. The \textit{height} of $\Phi$ is defined to be the least distance between a curve in $\Phi_*(\mu_1)$ and a curve in $\mu_2$, where distance is measured in the curve graph of $\Sigma_2$. Brock--Minsky--Namazi--Souto also introduce a condition called $R$\textit{-bounded-combinatorics}, which is a restriction on the complexity of the gluing maps \cite[Section 2.3]{bmns16}. The bounded combinatorics condition has two parts, one designed to guarantee that the closed geodesics corresponding to the gluing mapping classes stay in the thick part of moduli space, and another designed to ensure that boundaries of compressing disks are not too close to the stable foliations of the gluing maps. Since the decorated manifolds we glue together will have incompressible boundary, only the former condition will matter for us. 

For each $\mathcal{M}$-gluing $X$ with $R$-bounded combinatorics, there is a $\textit{model metric}$, denoted by $\mathcal{M}_X$.  This metric is built from a fixed metric on each element of $\mathcal{M}$ together with metrics on the product regions interpolating between paired boundary components. The metric on each manifold in $\mathcal{M}$ is chosen once and for all so that the induced metric on each of its boundary components is a hyperbolic metric.  Besides that, there are no conditions it needs to satisfy. For a boundary pairing $\Phi: \Sigma_1 \to \Sigma_2$, let $g_1$ and $g_2$ be the model metrics on $\Sigma_1$ and $\Sigma_2$ respectively. Let $\gamma:[0,L] \to \text{Teich}(\Sigma_1)$ be the unit speed parameterization of the geodesic in Teichm\"uller space in the Teichm\"uller metric between $g_1$ and $\Phi^*(g_2)$.  Let $g_\Phi$ be the metric on $\Sigma \times [0,L]$ defined by the formula $$g_\Phi(x,t) = dt^2 + \gamma(t)(x).$$ The model metric on $X$ is obtained by gluing the surface $\Sigma \times 0\subset \Sigma \times [0,L]$ to $\Sigma_1$ via the identity map and $\Sigma \times L$ to $\Sigma_2$ by $\Phi$. By construction, both of these gluings are isometries. Finally, we smooth the corners and call the resulting metric on $\mathcal M_X$ the model metric.

The main theorem of \cite{bmns16} is the following: 

\begin{thm} [{\cite[Theorem 8.1]{bmns16}}] \label{bmnsmainthm}
    Let $\mathcal{M}$ be a finite collection of decorated  irreducible, atoroidal 3-manifolds $M$ whose fundamental groups are nonabelian, and fix $R>0$.  Then there exist $D$ and $K$ such that, for any $\mathcal{M}$-gluing $X$ with $R$-bounded combinatorics and all heights greater than $D$, $X$ admits a unique hyperbolic metric $\sigma$.  Moreover, there exists a $K$-bilipschitz diffeomorphism from the model $\mathcal{M}_X$ to $(X,\sigma)$ in the correct isotopy class, whose image is the complement of the rank 2 cusps in $X$.
 \end{thm}
 
The manifolds we construct in this paper will not have cusps, so the image of the $K$-bilipschitz diffeomorphism given by the previous theorem will be all of $X$. The notion of $R$-bounded combinatorics simplifies greatly in the case of incompressible boundary components: 

\begin{prop}\label{prop:rbounded}
    Suppose that all of the manifolds in $\mathcal M$ have incompressible boundary and none are $I$-bundles over a surface. 
    Suppose $\Phi:\Sigma_1\to\Sigma_2$ is a pairing between boundary components of two of the decorated manifolds in $\mathcal M$. Let $\phi:\Sigma_1 \to \Sigma_1$ be a mapping class. Then the sequence of boundary pairings $\Phi\circ \phi^i$ has $R$-bounded combinatorics for some uniform $R>0$ independent of $i$. If $\phi$ is pseudo-Anosov, then the heights of $\Phi\circ\phi^i$ go to infinity as $i \to \infty$.
\end{prop}

\begin{proof}
The closed Teichmuller geodesic corresponding to the mapping class $\phi^i$ is contained in the $\epsilon$-thick part of moduli space for some $\epsilon$ independent of $i$. It follows that the first condition of $R$-bounded combinatorics holds for some $R$ independent of $i$. Since we are gluing along incompressible surfaces, the second condition of $R$-bounded combinatorics relating to compressing disks is trivially satisfied. That heights tend to infinity is a consequence of the fact that the pseudo-Anosov map $\phi$ acts as a hyperbolic isometry on the curve graph, which is $\delta$-hyperbolic (see \cite[Sections 2.1,2.3]{mah10}.)  
\end{proof}

In this paper, we will often choose $\phi$ to be a pseudo-Anosov mapping class acting trivially on homology. Then by taking $i$ large, we will choose the gluing heights as large as necessary for \cref{bmnsmainthm} above to apply.

\end{subsection}

\end{section}

\begin{section}{Construction}\label{sec:construction}

\begin{subsection}{Constructing hyperbolic homology sums of handlebodies}
Say that a compact orientable 3-manifold $B$ is a \emph{hyperbolic homology sum of $n$ handlebodies}, or an \emph{HHH}, if 
\begin{enumerate}
	\item $B$ is irreducible, atoroidal, and has incompressible boundary. \textit{We do however allow essential annuli.}  That is to say, we do not require $B$ to be acylindrical.  
	\item {$\partial B \cong \Sigma_1 \sqcup \dots \sqcup \Sigma_n$ where each $\Sigma_i$ is a surface of fixed genus $g\geq 2$.} 
	\item $B$ is integer homology equivalent to the connect sum of $n$ handlebodies. In particular, $H^1(B,\Z) \cong \bigoplus_i L_i$, where $L_i$ is a half dimensional subspace of $H^1(\Sigma_i, \Z)$.
\end{enumerate}
We refer the reader to \cite{bonsurvey} for facts about JSJ decompositions, geometrization theorems, and relations between them that we will freely use in what follows. The goal of this section is to construct such a manifold for any $g\geq 2$, $n\geq 1$. Here is our strategy. Start with the connect sum of $n$ genus $g$ handlebodies. Choose a complicated null-homologous hyperbolic knot and do a long Dehn surgery along this knot. A well chosen Dehn surgery doesn't change the integer homology, but it makes the 3-manifold hyperbolic.

Let us now construct an explicit example. First note that the connected sum of $n$ handlebodies is homeomorphic to the manifold obtained by removing $n$ standard solid handlebodies from $S^3$. Let $K$ be a hyperbolic knot in $S^3$. Let $F_g$ be the free group on $g$ generators. Find $n$ injections $\phi_1,\dots ,\phi_n$ of $F_{g}$ into the commutator subgroup of $\pi_1(S^3 \setminus K)$. We do not require that the injections have disjoint images. This can be done using the ping-pong lemma, which says that if $g,h$ are two hyperbolic elements of $PSL_2(\C)$ which do not commute, then for $k$ sufficiently large, $g^k$ and $h^k$ generate a free group. This locates a free group of rank 2 in the commutator subgroup of $\pi_1(S^3 \setminus K)$, which in turn contains a free group on $g$ generators. 

Let $W_i$, $1\leq i \leq n$ be wedge sums of $g$ circles. Choose embeddings of $W_i$ in $S^3 \setminus K$ realizing the homomorphism $\phi_i:\pi_1(W_i) \to \pi_1(S^3\setminus K)$. We can arrange that after ignoring $K$, the image of $\cup_i W_i$ in $S^3$ is isotopic to the standard unlinked embedding of $n$ wedges of $g$ circles in $S^3$. One way to achieve this begins with an ordering on the $ng$ loops in $\cup_i W_i$. Then one draws $W_1\cup \dots \cup W_n \cup K$ as a knot diagram in the plane, and changes the crossings between loops so that loops earlier in the ordering always pass above loops later in the ordering. During this procedure, we do not change any of the crossings involving $K$, so we do not change the homomorphisms $\phi_i$. However, we unknot and unlink the $W_i$'s.

\newcommand{\G}{\overline{W}}

Let $\G_i$ be a handlebody obtained as a tubular neighbourhood of $W_i$. Let $B^\circ= S^3 - (K \cup_i \G_i)$. We claim that $\partial(\G_i)$ is $\pi_1$-injective in $B^\circ$.  By the loop theorem, it is enough to show that $\partial(\G_i)$ is incompressible in $B^\circ$-- that there is no essential simple closed curve in the kernel of the inclusion of $\pi_1(\partial(\G_i))$ in $\pi_1(B^\circ)$. Suppose $\gamma \subset \partial \G_i$ is such a simple closed curve. By Dehn's lemma, $\gamma$ bounds a disk $D$ in $B^\circ$.  Since $\pi_1(\G_i)$ injects into $\pi_1(S^3\setminus K)$,  $\gamma$ must be null-homotopic in $\G_i$. Another application of Dehn's lemma gives that $\gamma$ bounds an embedded disk $D'$ in $\G_i$. 

Gluing $D'$ and $D$ together along $\gamma$, we obtain an embedded sphere $S\subset S^3 \setminus K$ such that $S\cap \partial(\G_i)= \gamma$.  Therefore $\gamma$ is a separating curve in $\partial(\G_i)$. Moreover, $\gamma$ is essential in $\partial \G_i$, so $S$ cuts $\G_i$ into two pieces each with nontrivial $\pi_1$. {Since $K$ is disjoint from $S$, it is contained entirely in one of the connected components of the complement of $S$. So the other connected component of $S^3 \setminus (K \cup S)$ is a ball containing some piece of $ \G_i$ with nontrivial $\pi_1$. But this contradicts the fact that $\pi_1(\G_i)$ injects into $\pi_1(S^3 \setminus K)$.  }

Next, we claim that $B^\circ$ has no non-peripheral essential tori. Suppose $T$ is a non-peripheral essential torus in $B^\circ$.  Consider how $T$ sits inside $S^3 \setminus K$. Since $K$ is a hyperbolic knot, either $T$ bounds a solid torus in $S^3 \setminus K$ or $T$ is the boundary of a tubular neighbourhood of $K$. In the first case, the component of $S^3\setminus K$ inside $T$ has fundamental group $\Z$, while in the latter case the component has fundamental group $\Z^2$. In either case, the $\G_i$'s cannot reside inside this solid torus, else they could not be $\pi_1$-injective in $S^3 \setminus K$. So the $W_i$'s live outside $T$. Consequently $T$ is either inessential or peripheral and parallel to $K$. One can check in a similar way that $B^\circ$ contains no essential embedded spheres. We have consequently shown that $B^\circ$ is irreducible and atoroidal. Therefore, $B^\circ$ has a JSJ decomposition into hyperbolic pieces. That is, there is a canonical collection of essential annuli $A$ along which we can cut so that the resulting manifolds are hyperbolic. Note that none of these essential annuli touch $K$, since such an annulus would give rise to a free homotopy between a curve in $W_i$ and a peripheral curve in $S^3 \setminus K$, and we constructed $W_i$ so that such free homotopies do not exist.

Let $B^\circ(K,\frac 1 q)$ denote the result of $\frac 1 q$ Dehn filling along $K$. This is our candidate for $B$, and $\Sigma_i = \partial \G_i$. For large enough $q$, this Dehn filling preserves the hyperbolicity of the components of the JSJ decomposition of $B^\circ$.  This is clear because $K$ is entirely contained in one of the components. Therefore, $B^\circ(K,\frac 1 q)$ remains atoroidal and irreducible. We wish to show that $\partial \G_i$ remains incompressible after Dehn filling $K$. For each $i$ let $B^\circ_i = S^3 - (K \cup W_i)$. The argument from the previous paragraph also shows that $B^\circ_i$ is irreducible and atoroidal. Although $B^\circ$ may contain essential annuli, $B^\circ_i$ is acylindrical and therefore hyperbolic. Let $DB_i^\circ$ be the double of $B^\circ_i$ along $\Sigma_i$. Since we are gluing $B^\circ_i$ to itself along a single incompressible boundary component, the result is still hyperbolic.  The resulting manifold has a $\Z/2$ symmetry fixing $\Sigma_i$. By Mostow rigidity, this symmetry is realized as an orientation reversing isometry, ie a reflection. Therefore, $\Sigma_i$ is totally geodesic in $DB^\circ_i$. Let $DB_i$ be the result of doing $\frac 1q$ and $-\frac 1q$ surgery along the two copies of $K$ in $DB^\circ_i$ for some large $q\in \Z$. For $q$ large enough, Thurston's hyperbolic Dehn surgery theorem says that the resulting manifold is hyperbolic. Moreover, the Dehn surgery respects the $\Z/2$ symmetry. Thus, $\Sigma_i$ remains totally geodesic in the Dehn filling and must be $\pi_1$-injective. It follows that $\Sigma_i$ is incompressible in $B^\circ_i(K,\frac 1 q)$, and therefore also incompressible in $B^\circ(K,\frac 1 q)$.

By Mayer--Vietoris, slope $\frac 1 q$ Dehn surgery on any knot in $S^3$ results in an integer homology sphere. The $W_i$'s have trivial linking number in $S^3$ with each other and with $K$ (since $\varphi_i$ was chosen to map into the commutator subgroup of $\pi_1(S^3\setminus K)$), so this remains the case after doing $\frac 1 q$ Dehn surgery on $K$. Therefore, \cref{prop:lh} below proves that the homology of $B$ has the desired form.

\begin{prop}\label{prop:lh}
    Suppose $M$ is an integer homology 3-sphere. Let $\G_1,\dots, \G_n$ be $n$ disjoint genus $g$ handlebodies embedded in $M$. Suppose that the linking form on $H_1(\G_i) \otimes H_1(\G_j)$ vanishes for all $i \neq j$. Let $X = M \setminus (\cup_i \G_i)$. Let $\Sigma_1\dots \Sigma_n$ be the genus $g$ boundary components of $X$. Then $H^1(X,\Z) \cong \bigoplus_{i=1}^n L_i$, where $L_i$ is a $g$ dimensional subspace of $H^1(\Sigma_i)$.
\end{prop}
\begin{proof}
    The Mayer-Vietoris sequence for the triple $M,\cup_i \G_i, X=M\setminus \cup_i \G_i$ reads
    $$H^1(M) \to H^1(X) \oplus H^1(\cup_i \G_i) \to H^1(\sqcup_i \Sigma_i) \to H^2(M)$$
    We assumed that $M$ is an integer homology 3-sphere, so $H^1(M)=H^2(M)=0$ and therefore the middle map is an isomorphism. So $H^1(X) \cong \Z^{ng}$.

    A similar argument shows that for each $i$, there is an isomorphism $H^1(M\setminus \G_i) \oplus H^1(\G_i) \cong H^1(\Sigma_i)$. By our assumption on the vanishing of the linking form, whenever $i\neq j$, each element of $H^1(M \setminus \G_i)$ restricts to zero in $H^1(\Sigma_j)$. So for each $i$, $H^1(M \setminus \G_i)$ is a $\Z^g$ summand of $H^1(X)$. These are the summands $L_i$ that we sought.
\end{proof}
\end{subsection}

\begin{subsection}{Assembling the pieces}
Thanks to the previous section, we have many hyperbolic homology sums of handlebodies at our disposal. We can glue these building blocks together to obtain closed hyperbolic 3-manifolds $M$.

Let $B_1,\dots, B_N$ be copies of a hyperbolic homology connect sum of two genus $g$ handlebodies. Each of these manifolds has two boundary components homeomorphic to $\Sigma$. Let $B_0$ and $B_{N+1}$ be hyperbolic homology genus $g$ handlebodies.

Choose identifications between the boundaries of adjacent blocks so that the resulting 3-manifold is an integer homology sphere. To be more precise, each piece $B_i$ looks homologically like a connect sum of two handlebodies. Glue adjacent $B_i$ as in a Heegaard decomposition of $S^3$, so that the entire manifold looks homologically like a connect sum of copies of $S^3$. Choose a pseudo-Anosov element $\phi$ of the Torelli group (the subgroup of the mapping class group of the surface acting trivially on homology) and twist all the gluing maps by $\phi^k$, where $k$ is a big number to be determined later.  Such a $\phi$ exists by \cite{fm11}.  

Let $\Sigma_i$ be the surface along which $B_{i-1}$ and $B_i$ are glued.

Call the resulting 3-manifold $M$. Since we are gluing irreducible and atoroidal pieces along incompressible surfaces of genus $\geq 2$ the resulting 3-manifold is hyperbolic by the main theorem of \cite{bmns16}, provided we glue by a large enough power of $\phi$.

For the purposes of getting a spectral gap, the specific construction used is not too important. We will only use the properties below:
\begin{prop}\label{prop:hknd}
	For each $i$, the restriction map $$H^1(B_{i-1} \cup B_i \cup B_{i+1}, \R) \to H^1(B_i,\R)$$ is trivial.
\end{prop}
\begin{proof}
    The intersection between the images of the restriction maps $H^1(B_{i-1}) \to H^1(\Sigma_{i-1})$ and $H^1(B_{i}) \to H^1(\Sigma_{i-1})$ is zero, by our choice of gluing. Therefore,
    the restriction map $H^1(B_{i-1} \cup B_i) \to H^1(\Sigma_{i-1})$ is zero. Similarly, the restriction map $H^1(B_i \cup B_{i+1}) \to H^1(\Sigma_{i})$ is zero. Define $r$ and $s$ to be the maps below.
    $$H^1(B_{i-1} \cup B_i \cup B_{i+1}, \R) \xrightarrow{r} H^1(B_i,\R) \xrightarrow{s} H^1(\Sigma_{i-1}) \oplus H^1(\Sigma_i)$$
    Since $B_i$ is an integer homology connect sum of two handlebodies, $s$ is injective. We have just shown that $s \circ r=0$. So $r=0$, as desired.
\end{proof}
\cref{prop:hknd} says that all homology in $B_i$ is killed locally.
\begin{prop}
    For sufficiently large $k$ and each $1\leq i \leq N$, there is a subregion of $M$ corresponding to $B_i$ that is bilipschitz to a standard model independent of $i,N$. For $B_0$ and $B_{N+1}$, there are subregions of $M$ bilipschitz to standard models independent of $N$. All bilipschitz constants are uniform in $i,N$.  The subregions of $M$ corresponding to $B_0,B_1,\dots,B_{N+1}$ cover all of $M$.  

\end{prop}
\begin{proof}
    Since all our gluing identifications are twisted by some Torelli element $\varphi^k$ (for $k$ large and independent of $i,N$), \cref{prop:rbounded} ensures that our gluing satisfies the conditions of \cref{bmnsmainthm}.
\end{proof}

{Note that the manifolds constructed in this section have unbounded rank and do not converge to a tame manifold.}

\end{subsection}

\end{section}

\begin{section}{Variants of the construction}\label{sec.var}
\begin{subsection}{Tower of covers}
The construction can be modified so that the sequence of manifolds is a tower of covers. {We will find a 2-boundary-component HHH $\widetilde B$ that double covers a manifold $B$ that is nearly a 1-boundary-component HHH. We say ``nearly'' because $B$ will have some additional homology associated with the $\Z/2$ cover; the best we can do is to arrange that $B$ will be a hyperbolic homology connect sum of a handlebody and $\mathbb{RP}^3$. As in the previous section, glue two copies of $B$ in such a way that the resulting manifold is a homology $\mathbb{RP}^3\#\mathbb{RP}^3$, twisting the gluing identification by a large power of a pseudo-Anosov Torelli element $\phi^k$ to ensure that the result is a hyperbolic 3-manifold. The resulting 3-manifold admits a tower a covers by repeatedly unfolding one of the $B$'s. The $n^{th}$ manifold in the tower looks like $2^n-1$ copies of $\widetilde B$ glued end to end, capped off with a copy of $B$ at the ends. The techniques of \cref{thm:1} apply equally well to show that this sequence has a uniform coexact spectral gap.}

Now let's construct $B$ and $\widetilde B$. Let $L$ be the link in \cref{fig:rp3_knot}. We verified that $L$ is hyperbolic using SnapPy. The antipodal map on $S^3$ exchanges the two components. This is a consequence of some simpler order 2 symmetries of the diagram. A $180^\circ$ degree rotation of the diagram interchanges the two components. Circular inversion in the plane (or equivalently, spherical inversion with respect to a unit sphere centred at the middle of the diagram) sends the diagram to itself, but changes every overcrossing to an undercrossing. Finally, reflection in the plane of the diagram changes every over crossing to an undercrossing. The antipodal map on $S^3$ is the product of these three maps, and so has the claimed behaviour. 

Therefore, $L$ descends to a null-homologous hyperbolic knot $K$ in $\mathbb{RP}^3$. The linking number between the two components of $L$ is 0. Therefore, for any $n>0$, the result of simultaneous $1/n$ Dehn surgery on the two components of $L$ is an integer homology sphere. 

As in \cref{sec:construction}, excise a null-homologous, $\pi_1$-injective wedge sum of circles $W$ from $\mathbb{RP}^3\setminus K$. We also want $W$ to be standardly embedded in $\mathbb{RP}^3$ when ignoring $K$. The construction in \cref{sec:construction} works here. Since $W$ is embedded in $\mathbb{RP}^3$ in a standard way, $W$ is contained in a ball in $\mathbb{RP}^3$. So $W$ lifts to two null-homologous wedge sums of circles $W_1$ and $W_2$ in $S^3 \setminus L$, and moreover $W_1$ and $W_2$ are contained in disjoint balls in $S^3$ and hence are standardly embedded in $S^3$. Also, $W_1$ and $W_2$ are both $\pi_1$-injective in $S^3\setminus L$.

Now as argued in \cref{sec:construction}, for large enough $n$, $1/n$ Dehn surgery on $L$ in $S^3 \setminus (W_1 \cup W_2)$ results in a hyperbolic homology sum of two handlebodies. We can take this manifold to be $\widetilde B$, and take $B$ as its quotient by the antipodal map.

\begin{figure}[htbp]
\centerline{\includegraphics[width=0.4\textwidth]{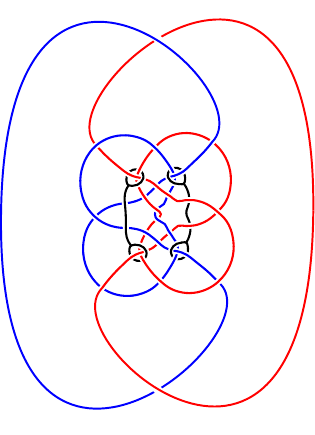}}
\caption{The link $L$ is shown in red and blue. A possibility for $W_1$ and $W_2$ is drawn in black.}\label{fig:rp3_knot}
\end{figure}
\end{subsection}

\begin{subsection}{Expander graphs}
We can arrange that the sequence of manifolds also has a spectral gap for the Laplacian acting on functions, at the cost of allowing our manifolds to have nonzero rational first homology. This will prove \cref{thm.13}.

A $d$-regular expander family is an infinite family of connected $d$-regular graphs $\mathcal G$ so that for any $G=(V,E)$ in $\mathcal G$ and any mean-zero function $f \in \ell^2(V)$, we have 
$$ \sum_{v\in V} |f(v)|^2 \leq \frac1{c^{2}} \sum_{v\in V} \left|\sum_{v'\sim v} \left(f(v') - f(v)\right) \right|^2 $$
where $c>0$ is a uniform constant (depending on the family). 

Now fix an expander family of $d$-regular graphs $\mathcal G$.
Let $B$ be a hyperbolic homology connect sum of $d$ handlebodies. For any $G\in \mathcal G$, we can form a hyperbolic $3$-manifold $M$ by taking a copy of $B$ for each vertex $v$ of $G$ and gluing the boundary components along edges with a large, fixed power of a suitable pseudo-Anosov mapping class. 

\begin{prop}
    There is a {uniform} spectral gap for functions on the 3-manifold $M$, i.e., for any mean-zero function $g \in L^2(M) \cap C^1(M)$, {one has 
    $$ \int_{M} |g|^2 \lesssim \int_M |dg|^2\,, $$
    where the implicit constant is independent of the particular graph $G$ in the expander family.}
\end{prop}
\begin{proof}
    The construction of $M$ and Theorem~\ref{bmnsmainthm} implies that we can write $M = \bigcup_{v\in V} \tilde B_v$, where each $\tilde B_v$ is $K$-bilipschitz to a standard model $B$. {The model $\mathcal{M}_M$ of Theorem~\ref{bmnsmainthm} can be written as $\mathcal{M}_M = \bigcup_{v\in V} B_v$ where the $B_v$ are all isometric to $B$. Now, due to the uniform $K$-bilipschitz property, a uniform function spectral gap for the model $\mathcal{M}_M$ implies one for the original $M$. Thus, it suffices to show a uniform spectral gap for the model $\mathcal{M}_M$.} 
    Now we define a map $T : L^2(\mathcal{M}_M) \to \ell^2(G)$ by 
    $$ T(g)(v) := \frac1{|B_v|} \int_{B_v} g$$
    Notice that we have
    \begin{align*}
        \|g\|_{L^2(\mathcal{M}_M)}^2 = \sum_{v\in V} \int_{B_v} |g|^2 \geq \sum_{v\in V} {|B_v|} \left[\frac1{|B_v|} \int_{B_v} g\right]^2 = |B| \|T(g)\|_{\ell^2(G)}^2
    \end{align*}
    and so $T$ is a bounded linear operator. Note also that $T$ sends mean-zero functions to mean-zero functions. Now we claim that for adjacent vertices $v\sim v'$,
    \begin{equation}
        |T(g)(v')-T(g)(v)|^2 \lesssim \int_{B_{v'}\cup B_v} |\nabla g|^2.
    \end{equation}
    To show the claim, let $\bar g := |B_v\cup B_{v'}|^{-1} \int_{B_v\cup B_{v'}}g$ and note that 
    $$T(g)(v')-T(g)(v) = T(g-\bar g)(v')-T(g-\bar g)(v).$$ Thus we get
    \begin{align*}
        |T(g)(v')-T(g)(v)|^2 &= |T(g-\bar g)(v')-T(g-\bar g)(v)|^2\\
        &\leq2\left(T(g-\bar g)(v')^2 + T(g-\bar g)(v)^2\right)\\
        &\leq 2 |B|^{-1} \int_{B_{v'}\cup B_v} |g-\bar g|^2 \leq 2 |B|^{-1} p_B \int_{B_{v'}\cup B_v} |\nabla g|^2,
    \end{align*}
    where in the last inequality we used the Poincar\'e inequality on $B_{v'}\cup B_v$, and where $p_B$ is the constant in the Poincar\'e inequality which, importantly, is independent of $v,v'$.
    Thus the claim is proved.
    {Now let $c_G > 0$ denote the spectral gap for $G$. Since $G$ varies in an expander family, $c_G$ is bounded below by a uniform constant.} We complete the proof by writing
    \begin{align*}
        \|g\|_{L^2(\mathcal{M}_M)}^2 &= \sum_{v\in V} \int_{B_v} |g|^2 \leq \sum_{v\in V} \int_{B_v} |g - T(g)(v)|^2 + \sum_{v\in V} |B| T(g)(v)^2\\
        &\leq \sum_{v\in V} p_B \int_{B_v} |\nabla g|^2 + |B| c_G^{-2} \sum_{v\in V} \left|\sum_{v'\sim v} \left(T(g)(v') - T(g)(v)\right) \right|^2\\
        &\lesssim p_B \|\nabla g\|_{L^2(\mathcal{M}_M)}^2 + |B|c_G^{-2} p_B \sum_{v\in V} \sum_{v'\sim v} \int_{B_{v'}\cup B_v} |\nabla g|^2
        \\&\lesssim \|\nabla g\|_{L^2(\mathcal{M}_M)}\,.
        \qedhere
    \end{align*}
\end{proof}

Recall that the cohomology of $B$ takes the form $L_1 \oplus \dots \oplus L_d$, where $L_d$ is a Lagrangian subspace of the cohomology of the $i^{th}$ boundary component. Choose the gluings so that the Lagrangian subspaces on identified boundary components are transverse; this guarantees that the analogue of \cref{prop:hknd} holds:
for each $v\in V$ the restriction map
$$ H^1\left(B_v \cup \textstyle{\bigcup}_{v'\sim v} B_{v'}; \mathbb{R}\right) \to H^1(B_v;\mathbb{R})$$
is trivial. Then by an argument analogous to the proof of Theorem~\ref{thm:1}, {we will see that }the 3-manifolds $M$ thus constructed will have a coexact $1$-form spectral gap in addition to a function spectral gap.

\end{subsection}

\end{section}
\section{Analytic preliminaries: finding bounded primitives}\label{sec.bddprim}

In this section, let $X$ be a compact smooth manifold of dimension $n$ with boundary. All implicit constants may depend on $X$.
The results presented here are likely standard, but we weren't able to find an exact reference for our applications, so we give full details for the convenience of the reader.
	
	\begin{prop}[Non-compactly supported case] \label{prop:non-cpt}
		Let $\alpha$ be an $r$-form on $X$ which is exact in the sense that there is a smooth $(r-1)$-form on $X$ whose exterior derivative is $\alpha$. Then there is a smooth $(r-1)$-form $\beta$ defined in the interior of $X$ with $d\beta = \alpha$ and
		\begin{align} \label{eqn:non-cpt_primitive_bound}
			\|\beta\|_{W^{1,2}(X)}
			\lesssim \|\alpha\|_{L^2(X)}.
		\end{align}
	\end{prop}
	
		From the proof one can extract other Sobolev and H\"older bounds on $\beta$ of the quality one would expect from elliptic regularity (despite the fact that $d$ is not elliptic). We only need an $L^2 \to L^2$ bound though, so \eqref{eqn:non-cpt_primitive_bound} is already more than enough.
	
	\begin{proof}
		The strategy is to reduce to the case of a closed manifold by embedding $X$ in its double, and then to use Hodge theory. It's rather miraculous that this works (see the comment after the claim below).
		
		Fix a collar neighborhood $V \simeq \partial X \times [0,1)$ of the boundary. Let $X^-$ be a copy of $X$, and let $V^- \subseteq X^-$ be a copy of $V$, but this time identify $V^- \simeq \partial X \times (-1,0]$ by negating the $[0,1)$ coordinate on $V$. Let $\widetilde{X} = X^- \sqcup_{\partial X} X$ be the double of $X$ obtained by gluing $X^-,X$ together along their common boundary. This is a closed smooth manifold, where the smooth structure on $V^- \sqcup_{\partial X} V \subseteq \widetilde{X}$ comes from identifying $V^- \sqcup_{\partial X} V \simeq \partial X \times (-1,1)$. Given a form $\eta$ on $X$, let $\eta^-$ denote its copy on $X^-$, and $\tilde{\eta}$ the extension of $\eta$ to $\widetilde{X}$ given by setting $\tilde{\eta} = \eta^-$ in $\Int X^-$. Note that even if $\eta$ is smooth, $\tilde{\eta}$ may be discontinuous on $\partial X$.
		
		By assumption, there exists $\gamma \in C^{\infty}(X, \wedge^{r-1} T^*X)$ with $d\gamma = \alpha$.
		
		\begin{claim*}
			$d\tilde{\gamma} = \tilde{\alpha}$ in the sense of distributions on $\widetilde{X}$.
		\end{claim*}
		
		The miracle here is that $d\tilde{\gamma}$ does not have a singular part supported on $\partial X$, even though it is discontinuous on $\partial X$. The reason for this is that the singular parts of $d\gamma$ and $d \gamma^{-}$ cancel each other. To be pedantic, here $d\gamma$ denotes the exterior derivative of the extension by zero of $\gamma$ to $\widetilde{X}$, and similarly for $d\gamma^-$; of course the exterior derivative on $X$ of $\gamma$ is smooth.
		
		\begin{proof}[Proof of Claim]
			The statement that $d\tilde{\gamma} = \tilde{\alpha}$ as distributions is equivalent to the statement that for any orientable open subset $U \subseteq \widetilde{X}$ and any $\varphi \in C_c^{\infty}(U, \wedge^{n-r} T^*U)$,
			\begin{align} \label{eqn:d_dist_def}
				\int_U \tilde{\alpha} \wedge \varphi
				= (-1)^r \int_U \tilde{\gamma} \wedge d\varphi,
			\end{align}
			where the integrals are computed with respect to some common choice of orientation on $U$. We compute
			\begin{align*}
				\int_U \tilde{\alpha} \wedge \varphi
				&= \int_{U \cap X} \alpha \wedge \varphi + \int_{U \cap X^-} \alpha^- \wedge \varphi
				\\&= \int_{U \cap X} d\gamma \wedge \varphi + \int_{U \cap X^-} d\gamma^- \wedge \varphi
				\\&= \int_{\partial(U \cap X)} \gamma \wedge \varphi
				+ (-1)^r \int_{U \cap X} \gamma \wedge d\varphi
				+ \int_{\partial(U \cap X^-)} \gamma^- \wedge \varphi
				+ (-1)^r \int_{U \cap X^-} \gamma^- \wedge d\varphi.
			\end{align*}
			Since $U \cap X$ and $U \cap X^-$ both have boundary $U \cap \partial X$, but with opposite orientations, and $\gamma,\gamma^-$ pull back to the same form on $\partial X$, the boundary terms cancel. We conclude \eqref{eqn:d_dist_def}.
		\end{proof}

		Let $\Delta = dd^* + d^*d$ be the Hodge Laplacian on $\widetilde{X}$, defined with respect to some fixed choice of smooth Riemannian metric. By expanding the $L^2$ form $\tilde{\alpha}$ in an eigenbasis for $\Delta$, we can write
		\begin{align*}
			\tilde{\alpha} = h + \Delta\omega,
		\end{align*}
		where $h,\omega$ are $r$-forms, $h$ is harmonic, and $\omega$ is orthogonal to all harmonic forms on $\widetilde{X}$. By elliptic regularity,
		\begin{align} \label{eqn:omega_ell_reg}
			\|\omega\|_{W^{2,2}(\widetilde{X})}
			\lesssim \|\tilde{\alpha}\|_{L^2(\widetilde{X})}
			\sim \|\alpha\|_{L^2(X)}.
		\end{align}
		So far, we have the equations
		\begin{align} \label{eqn:Hodge_decomp_1}
			d\tilde{\gamma} = \tilde{\alpha} = h + d^*d\omega + dd^*\omega.
		\end{align}
		Until now, our $L^2$ norms have only really been defined up to constants, but from now on let us fix the $L^2$ inner product of forms on $\widetilde{X}$ to be the Hodge inner product $\langle \cdot,\cdot \rangle$ coming from our choice of metric. Then by integration by parts,
		\begin{align*}
			\|h\|_{L^2(\widetilde{X})}^2
			= \langle d(\tilde{\gamma} - d^*\omega), h \rangle - \langle d^*d\omega, h \rangle
			= 0,
		\end{align*}
		so $h = 0$. Therefore \eqref{eqn:Hodge_decomp_1} simplifies to
		\begin{align*}
			d\tilde{\gamma}
			= \tilde{\alpha}
			= d^*d\omega + dd^*\omega.
		\end{align*}
		Integrating by parts again (which is justified by approximating $\omega$ by smooth forms in the $W^{2,2}$ norm),
		\begin{align*}
			\|d^*d\omega\|_{L^2(\widetilde{X})}^2
			= \langle d(\tilde{\gamma} - d^*\omega), d^*d\omega \rangle = 0,
		\end{align*}
		so $d^*d\omega = 0$. Thus $\tilde{\alpha} = dd^*\omega$.
		
		Denote $\beta = d^*\omega|_{\Int X}$. Then $\beta$ is smooth, because $\omega$ is smooth in $\Int X$ by elliptic regularity. By the conclusion of the paragraph above, $d\beta = \alpha$. Finally, \eqref{eqn:omega_ell_reg} implies the desired bound \eqref{eqn:non-cpt_primitive_bound}.
	\end{proof}
	
	\begin{prop}[Compactly supported case] \label{prop:cpt}
		Let $\alpha \in L_c^2(\Int X, \wedge^r T^*X)$ be the exterior derivative of some form in $C_c^{-\infty}(\Int X, \wedge^{r-1} T^*X)$. Then there exists $\beta \in L_c^2(\Int X, \wedge^{r-1} T^*X)$ such that $d\beta = \alpha$ and
		\begin{align} \label{eqn:cpt_primitive_bound}
			\|\beta\|_{L^2(X)}
			\lesssim \|\alpha\|_{L^2(X)}.
		\end{align}
	\end{prop}

	Again, the proof could be tweaked to give better regularity for $\beta$, but we need not bother.
	
	\begin{proof}
		Fix some choice of smooth Riemannian metric on $X$, normalize the $L^2$ inner product of forms on $X$ to be the Hodge inner product $\langle \cdot,\cdot \rangle$ coming from this metric, and let $d^*$ be the formal adjoint of $d$ with respect to $\langle \cdot,\cdot \rangle$, as usual. Write $\alpha = d\gamma$ for some distributional form $\gamma$ supported in $\Int X$. Let $\ell$ be the linear functional on $d^*C_c^{\infty}(\Int X, \wedge^r T^*X)$ given by
		\begin{align} \label{eqn:l_def}
			\ell(d^*\varphi)
			= \langle d^*\varphi, \gamma \rangle
			= \langle \varphi,\alpha \rangle.
		\end{align}
		Fix an open neighborhood $U$ of the support of $\gamma$, such that $U$ is precompact in $\Int(X)$, and $U$ has smooth boundary. For each $\varphi$, the analog of Proposition \ref{prop:non-cpt} for $d^*$ instead of $d$ (which reduces to Proposition \ref{prop:non-cpt} by writing $d^* = \pm\star d \star$) yields a $\psi \in C^{\infty}(U, \wedge^{r-1}T^*U)$ with $d^*\psi = d^*\varphi$ and
		\begin{align*}
			\|\psi\|_{L^2(U)} \lesssim_U \|d^*\varphi\|_{L^2(U)}.
		\end{align*}
		Furthermore, if we choose $U$ to be almost all of $X$, so that there is a diffeomorphism between $\overline{U}$ and $X$ which is close to an isometry, then we can remove the dependence of the implicit constant on $U$. We now estimate
		\begin{align*}
			|\ell(d^*\varphi)|
			= |\langle d^*\varphi, \gamma \rangle|
			= |\langle d^*\psi, \gamma \rangle|
			= |\langle \psi,\alpha \rangle|
			\leq \|\alpha\|_{L^2(X)} \|\psi\|_{L^2(U)}
			\lesssim \|\alpha\|_{L^2(X)} \|d^*\varphi\|_{L^2(U)}.
		\end{align*}
		Using Hahn--Banach, extend $\ell$ to a linear functional on $C_c^{\infty}(\Int X, \wedge^{r-1} T^*X)$ while preserving the bound
		\begin{align} \label{eqn:l_bd}
			|\ell(\eta)|
			\lesssim \|\alpha\|_{L^2(X)} \|\eta\|_{L^2(U)}.
		\end{align}
		This functional is represented by a distribution $\beta \in C^{-\infty}(\Int X, \wedge^{r-1} T^*X)$. It follows from \eqref{eqn:l_def} that $d\beta = \alpha$, and from \eqref{eqn:l_bd} we obtain \eqref{eqn:cpt_primitive_bound} as well as the fact that $\beta$ is supported on the compact subset $\overline{U}$ of $\Int X$.
	\end{proof}

\begin{section}{Proof of \texorpdfstring{\cref{thm:1}}{thmone}}\label{sec.pf:1}

If $d\beta$ = $\alpha$, we say that $\beta$ fills $\alpha$. If $d^* \beta=\alpha$, we say that $\beta$ cofills $\alpha$. We usually want to find fillings or cofillings of small norm. Recall the following characterizations of the spectral gap:

\begin{lem}\label{lem:dualchar}
    The following are equivalent for a closed 3-manifold:
    \begin{enumerate}
        \item Every exact (resp. closed) 2-form $\alpha$ has a filling of norm $\leq c\norm{\alpha}_{L^2}$.
       \item Every coexact (resp. coclosed) 1-form $\gamma$ has a cofilling of norm $\leq c\norm{\gamma}_{L^2}$.
        \item $\norm{d\gamma}_{L^2} \geq \frac 1 c\norm{\gamma}_{L^2}$ for every coexact (resp. coclosed) 1-form $\gamma$.
        \item The spectral gap for the Laplacian acting on coexact (resp. coclosed) 1-forms is $\frac 1 {c^2}$.
    \end{enumerate}
\end{lem}

As in the previous section, $\|\cdot\|_{L^2}$ is the Hodge norm on forms.

\begin{proof}
    $(1) \iff (2)$ because the Hodge star preserves norms.
    $(3) \iff (4)$ because both are characterizations of the least nonzero singular value of $d$. (3) is the variational characterization and (4) is the characterization as the smallest nonzero eigenvalue of $d^*d$.

    Let's first prove $(2) \implies (3)$. Suppose $\gamma$ is a coexact (resp. coclosed) 1-form. Assuming (2), we can find a cofilling $\beta$ for $\gamma$ with $\norm{\beta}_{L^2} \leq c \norm{\gamma}_{L^2}$. Then
    \begin{align*}
        \norm{\gamma}_{L^2}^2 &= \inner{\gamma}{d^*\beta}\\
        &= \inner{d\gamma}{\beta} \\
        &\leq \norm{d\gamma}_{L^2} \cdot \norm{\beta}_{L^2} \\
        &\leq \norm{d\gamma}_{L^2} \cdot c\norm{\gamma}_{L^2}.
    \end{align*}
    Therefore
    \begin{align*}
        \norm{d \gamma}_{L^2} &\geq \frac 1 c \norm{\gamma}_{L^2}.
    \end{align*}

    Now we prove $(4) \implies (2)$. First we show how to cofill the eigenforms for the Laplacian. If $\gamma$ is a coexact (resp. coclosed) eigenform of the Laplacian of eigenvalue $\lambda\geq 1/c^2$, then $\frac 1 \lambda d\gamma$ is a cofilling of $\gamma$, and
    \begin{align*}
        \sqnorm{\frac 1 {\lambda} d\gamma}_{L^2} &= \frac 1 {\lambda^2} \inner {d\gamma}{d\gamma}\\
        &= \frac 1 {\lambda^2} \inner {d^* d \gamma}{\gamma}\\
        &= \frac 1 \lambda \sqnorm{\gamma}_{L^2} \\
        &\leq c^2 \sqnorm{\gamma}_{L^2}.
    \end{align*}
    So $\frac 1 \lambda d\gamma$ is an efficient cofilling. Any other 1-form can be expressed as a sum of orthogonal eigenforms, and hence may also be efficiently cofilled.
\end{proof}

 Our strategy for proving \cref{thm:1} is local to global; we prove local cofilling inequalities and deduce global cofilling inequalities. \cref{lem:splitting} and \cref{lem:filling} below show how to find some efficient local cofillings.


\begin{lem}[Splitting lemma]\label{lem:splitting}
    Suppose we have a Riemannian metric on $\Sigma \times (-\eps,\eps)$, a collar neighborhood of a surface. Given an exact 2-form $\alpha$ on $\Sigma \times (-\eps,\eps)$, there is an exact 2-form $\alpha'$ such that $\alpha-\alpha'=d\beta$ for some $\beta$ of compact support, $\alpha'=0$ on $\Sigma \times (-\eps/3,\eps/3)$, and
    \begin{align*}
        \|\beta\|_{L^2}
        \lesssim \|\alpha\|_{L^2(\Sigma \times (-\varepsilon,\varepsilon))}
        \qquad \text{and} \qquad
        \|\alpha-\alpha'\|_{L^2}
        \lesssim \|\alpha\|_{L^2(\Sigma \times (-\varepsilon,\varepsilon))},
    \end{align*}
    where the implicit constants depend only on the geometry of $\Sigma \times (-\eps,\eps)$.
\end{lem}

\begin{proof}
    Applying \cref{prop:non-cpt} with $X= {\Sigma \times [-3\eps/4,3\eps/4]}$, 
    we can choose a $1$-form $\gamma$ on $\Sigma \times (-3\varepsilon/4, 3\varepsilon/4)$ with $d\gamma = \alpha$ and 
    $$ \norm{\gamma}_{L^2(\Sigma \times (-3\eps/4,3\eps/4))} \lesssim \norm{\alpha}_{L^2(\Sigma \times (-\varepsilon,\varepsilon))}. $$
    Choose a standard, smooth cut-off function $\chi: \Sigma \times (-\eps,\eps) \to [0,1]$ such that $\chi \equiv 1$ on $\Sigma \times (-\eps/3,\eps/3)$ and $\chi \equiv 0$ on $\Sigma \times \{t\}$ for $2\eps/3\leq |t| \leq1$. Let $\beta = \chi\gamma$ and
    $$ \alpha' = \alpha - d\beta. $$
    Then clearly $\beta$ is of compact support, and $\alpha' = 0$ on $\Sigma \times (-\varepsilon/3,\varepsilon/3)$. It remains to estimate $\|\beta\|_{L^2}$ and $\norm{\alpha-\alpha'}_{L^2}$. The desired estimates follow easily from the $L^2$ bound on $\gamma$ in terms of $\alpha$:
    \begin{align*}
        \|\beta\|_{L^2}
        = \|\chi\gamma\|_{L^2}
        \lesssim \|\alpha\|_{L^2(\Sigma \times (-\varepsilon,\varepsilon))}
    \end{align*}
    and
    \begin{align*}
        &\norm{\alpha-\alpha'}_{L^2}
        = \|d(\chi\gamma)\|_{L^2}
        \leq \|d\chi \wedge \gamma\|_{L^2} + \|\chi\alpha\|_{L^2}
        \lesssim \|\alpha\|_{L^2(\Sigma \times (-\varepsilon,\varepsilon))}.
        \qedhere
    \end{align*}
    
\end{proof}

Now we return our attention to the hyperbolic 3-manifolds $M = M_N$ from Section \ref{sec:construction} which are built out of a sequence of blocks $B_0,\dots,B_{N+1}$ glued together end to end. 
\begin{lem}[Local filling inequality]\label{lem:filling}
    For any $i$ and any closed 2-form $\alpha$ supported in the interior of $B_i$, there is a 1-form $\beta$ supported in the interior of $B_{i-1}\cup B_i \cup B_{i+1}$, such that $\alpha = d\beta$ and $\norm{\beta}_{L^2} \lesssim \norm{\alpha}_{L^2}$. The implicit constant here is uniform in $i$ and $N$.
\end{lem}
\begin{proof}
    By \cref{prop:hknd}, $\alpha$ vanishes in compactly supported cohomology of $B_{i-1}\cup B_i \cup B_{i+1}$. The hyperbolic metric $g_{\text{hyp}}$ on $B_{i-1}\cup B_i \cup B_{i+1}$ is $K$-bilipschitz to a model metric $g_{\text{model}}$, where $K$ is independent of $i$ and $N$, and $g_{\text{model}}$ ranges over only finitely many metrics as $i,N$ vary. More precisely, $g_{\text{model}}$ is independent of $N$, and there are three possible choices for $g_{\text{model}}$ depending on whether $B_{i-1} \cup B_i \cup B_{i+1}$ is on the left end, in the middle, or on the right end of the manifold $M$. Denote the $L^2$ norms with respect to $g_{\text{model}}$ and $g_{\text{hyp}}$ by $\|\cdot\|_{\text{model}}$ and $\|\cdot\|_{\text{hyp}}$, respectively.

    Apply \cref{prop:cpt} to the exact 2-form $\alpha$ with metric $g_{\text{model}}$. This gives a 1-form $\beta$ supported in $B_{i-1}\cup B_i \cup B_{i+1}$ with $d\beta= \alpha$ and
    $$\norm{\beta}_{\text{model}} \lesssim \norm{\alpha}_{\text{model}}.$$
    Since $g_{\text{model}}$ and $g_{\text{hyp}}$ are bilipschitz (uniformly in $i,N$), this is equivalent to
    \begin{align*}
        &\norm{\beta}_{\text{hyp}} \lesssim \norm {\alpha}_{\text{hyp}}.
        \qedhere
    \end{align*}
\end{proof}

\begin{proof}[Proof of \cref{thm:1}]    
    Let $\alpha$ be any exact 2-form on $M$. By characterization 1 in \cref{lem:dualchar}, it suffices to find a primitve $\beta$ for $\alpha$ with $\|\beta\|_{L^2} \lesssim \|\alpha\|_{L^2}$, where here and in the remainder of this proof, the implied constant depends only on the family $\{M_N\}_N$ and hence is independent of all parameters (in particular independent of $N$).

        For each $i$, let $\Sigma_i$ be a surface separating $B_i$ from $B_{i+1}$. By \cref{lem:splitting} applied to tubular neighbourhoods of all of the $\Sigma_i$ (which we can take to all be uniformly bilipschitz to some model tubular neighborhood), there is a 1-form $\beta_1$ supported in the union of the tubular neighborhoods such that
	\begin{enumerate}
		\item[(i)] $\alpha = d \beta_1$ pointwise in a neighbourhood of the $\Sigma_i$'s,
		\item[(ii)] $\norm {\beta_1}_{L^2} \lesssim \norm{\alpha}_{L^2}$,
		\item[(iii)] $\norm{d \beta_1}_{L^2} \lesssim \norm{\alpha}_{L^2}$.
	\end{enumerate}

    As written, \cref{lem:splitting} only gives the inequalities above with norms taken in the model metric. But since the model metric and the hyperbolic metric are uniformly bilipschitz, these inequalities remain true with the hyperbolic metric.

    Let $\alpha' = \alpha - d \beta_1$. By (iii) and the triangle inequality,
    $$
    \norm {\alpha'}_{L^2} \lesssim \norm{\alpha}_{L^2}.
    $$
    Observe that $\alpha'$ decomposes as a sum of disjoint lumps, each of which is supported in the interior of one of the $B_i$. Thus by \cref{lem:filling}, we can write $\alpha'=d\beta_2$ for some 1-form $\beta_2$ satisfying
    \begin{align*}
        \norm {\beta_2}_{L^2} \lesssim \norm {\alpha'}_{L^2} \lesssim \|\alpha\|_{L^2}.
    \end{align*}
    Now $\beta=\beta_1+\beta_2$ is a primitive for $\alpha$ with
    \begin{align*}
        \|\beta\|_{L^2}
        \leq \|\beta_1\|_{L^2} + \|\beta_2\|_{L^2}
        \lesssim \|\alpha\|_{L^2},
    \end{align*}
    as desired.
\end{proof}

\end{section}

\begin{section}{Spectral gap with tame limit}\label{sec.pf:11}

 In this section, we will isolate the specific properties of a sequence of 3-manifolds $M_N$ with bounded rank that will give us a spectral gap.{ We will also show there do indeed exist sequences that  satisfy these properties, which will prove Theorem~\ref{thm:11}.}
	
	Let $E^+,E^-$ be 3-manifolds with boundary $\Sigma$. Fix inclusions of collar neighborhoods
	\begin{align*}
		\iota^+ \colon \Sigma \times [0,1) \hookrightarrow E^+
		\qquad \text{and} \qquad
		\iota^- \colon \Sigma \times (-1,0] \hookrightarrow E^-
	\end{align*}
	identifying $\Sigma \times \{0\} \simeq \partial E^{\pm}$. Fix a diffeomorphism $\varphi$ of $\Sigma$, and let $\tilde{\varphi} \colon \Sigma \times \mathbb{R} \to \Sigma \times \mathbb{R}$ denote the map
	\begin{align*}
		\tilde{\varphi}(x,t) = (\varphi(x), t+1).
	\end{align*}

	Say that a family of metrics $g_i$ on a fixed manifold is \emph{commensurate} if there is a uniform constant $C \geq 1$ such that $C^{-1} g_i \leq g_j \leq Cg_i$ for all $i,j$.
	
	Let $(M_N,g_N)$ be a sequence of Riemannian 3-manifolds with diffeomorphic identifications
	\begin{align*}
		M_N \simeq E^- \sqcup_{\Sigma \times (-N-1,-N]} \Sigma \times (-N-1,N+1) \sqcup_{\Sigma \times [N,N+1)} E^+
	\end{align*}
	satisfying the properties (i),(ii) below. Here $E^+,E^-$ are glued to $\Sigma \times (-N-1,N+1)$ along the inclusions
	\begin{align*}
		&\iota_N^+ \colon \Sigma \times [N,N+1) \xrightarrow{\tilde{\varphi}^{-N}} \Sigma \times [0,1) \xrightarrow{\iota^+} E^+,
		\\&\iota_N^- \colon \Sigma \times (-N-1,-N] \xrightarrow{\tilde{\varphi}^N} \Sigma \times (-1,0] \xrightarrow{\iota^-} E^-.
	\end{align*}
	\begin{enumerate}
		\item[(i)] The metrics $g_N|_{E^+}$ are commensurate, as are $g_N|_{E^-}$.
		
		\item[(ii)] There is a commensurate family of metrics $g_{N,n}$ on $\Sigma \times (-1,1)$, ranging over all $N$ and all $|n| \leq N$, such that
		\begin{align*}
			g_N|_{\Sigma \times (n-1,n+1)}
			= \tilde{\varphi}_*^n g_{N,n}.
		\end{align*}
	\end{enumerate}
	
	These are the only geometric assumptions we need to make about the $M_N$, but we also need some cohomological assumptions on $\varphi,E^{\pm},\iota^{\pm}$. Denote $V = H_c^2(\Sigma \times \mathbb{R})$ (all cohomology groups have real coefficients). Let $V^{\pm} \subseteq V$ be the Lagrangian subspaces
	\begin{align*}
		& V^+ = \Ker(V \simeq H_c^2(\Sigma \times (0,1)) \xrightarrow{\iota^+_*} H_c^2(E^+)),
		\\& V^- = \Ker(V \simeq H_c^2(\Sigma \times (-1,0)) \xrightarrow{\iota^-_*} H_c^2(E^-)).
	\end{align*}
	Assume:
	\begin{enumerate}
		\item[(iii)] $H^2(E^{\pm}) = 0$.
		
		\item[(iv)] The vector space $V$ splits as a direct sum of eigenspaces for $\tilde{\varphi}_*$ with eigenvalues $\neq \pm 1$.
		
		\item[(v)] $V^+$ (resp. $V^-$) has zero intersection with the contracting (resp. expanding) subspace of $V$ determined by $\tilde{\varphi}_*$.
	\end{enumerate}

	\begin{thm} \label{thm:lott_type_gap}
		Under the assumptions (i)-(v), the spectrum of the Laplacian on coclosed 1-forms (or equivalently on closed 2-forms) on $M_N$ is uniformly bounded below when $N$ is sufficiently large.
	\end{thm}

    For an example of a sequence of manifolds that satisfy assumptions (i)-(v), we can let $E^- = E^+$ be hyperbolic homology genus 2 handlebodies as in Section \ref{sec:construction}, and $\varphi: \Sigma_2\to \Sigma_2$ a pseudo-Anosov diffeomorphism acting on the homology $H_1(\Sigma_2)$ by the symplectic matrix
    \begin{align*}
		\begin{pmatrix}
			4 & 2 & 3 & 0 \\
			2 & 2 & 0 & 3 \\
			1 & 0 & 2 & -2 \\
			0 & 1 & -2& 4
		\end{pmatrix}
	\end{align*}
    in the basis $H_1(\Sigma_2;\mathbb{Z}) = \langle a_1,a_2,b_1,b_2\rangle$ where $a_1, a_2$ form a basis for $\text{Ker}(H_1(\Sigma_2;\mathbb{Z})\to H_1(E^+;\mathbb{Z}))$ and $b_1, b_2$ form a basis for $\text{Ker}(H_1(\Sigma_2;\mathbb{Z}) \to H_1(E^-;\mathbb{Z}))$, where these maps are obtained by viewing $\Sigma_2$ as a Heegaard surface for the manifold $M_0$ obtained by gluing $E^+$ to $E^-$ as in the standard Heegaard decomposition of $S^3$. The manifolds $M_N$ are obtained from $M_0$ by composing the gluing map for the Heegaard splitting of $M_0$ along $\Sigma_2$ with $\varphi^{2N}$. 

    Now we verify the conditions (i)-(v): 
    Recall from the discussion in Section~\ref{sec.gluing} that the model metric $g_\Phi$ on $M_N \simeq E^-\cup \Sigma \times [-N,N]\cup E^+$ is such that the restrictions of $g_{\Phi}$ to $E^-,E^+$ are $K$-bilipschitz to the respective hyperbolic metrics, and 
    $$ g_\Phi|_{\Sigma \times [-N,N]} = d(\ell t)^2 + \gamma(\ell t)
    = \ell^2 dt^2 + \gamma(\ell t)\,, $$
    where $\gamma : [-\ell N,\ell N] \to \text{Teich}(\Sigma)$ is the unit speed Teichm\"uller geodesic corresponding to $\varphi^{2N}$ (so $\ell$ is the length of the Teichm\"uller geodesic corresponding to $\varphi$). Thus we see that
    $$ g_\Phi|_{\Sigma \times (n-1,n+1)} = \varphi_*^n g_\Phi|_{\Sigma \times (-1,1)}. $$
    So, for sufficiently large values of $N$, the gluing heights are large and Theorem~\ref{bmnsmainthm} implies that the hyperbolic manifolds $(M_N,g_N)$ are uniformly $K$-bilipschitz to the model manifold, and we have verified assumptions (i) and (ii).

    Note that assumption (iii) is clearly verified as $E^\pm$ are homology handlebodies. By Poincar\'e duality and our choice of basis, the matrix under consideration represents the action of $\tilde \varphi_*$ on $V = V^+ \oplus V^-$. It has eigenvalues $3 \pm 2\sqrt{2}$, each with multiplicity $2$. Now the eigenvectors of $\tilde \varphi_*$ split into the contracting and expanding eigenspaces respectively as
    \begin{align*}
    \begin{Bmatrix}
		\begin{pmatrix}
			2 \\
			-2\sqrt{2}-1 \\
			0 \\
			1 
		\end{pmatrix} ,
        \begin{pmatrix}
			-2\sqrt{2}+1\\
			2 \\
			1  \\
			0 
		\end{pmatrix}
    \end{Bmatrix},
    \begin{Bmatrix}
		\begin{pmatrix}
			2 \\
			2\sqrt{2}-1 \\
			0 \\
			1 
		\end{pmatrix} ,
        \begin{pmatrix}
			2\sqrt{2}+1\\
			2 \\
			1  \\
			0 
		\end{pmatrix}
    \end{Bmatrix}
        \,.
	\end{align*}
    We now can see that the assumptions (iv) and {(v) are satisfied}.
    
    Thus, assuming Theorem~\ref{thm:lott_type_gap} for now, the above sequence $M_N$ gives a proof of Theorem~\ref{thm:11}. We remark that the $M_N$ are rational homology 3-spheres with torsion homology
    $$H_1(M_N;\mathbb{Z})_{\text{tors}} = (\mathbb{Z}/A_{2N}\mathbb{Z})^2,$$
    where $A_i$ is given by the recurrence $A_0=0, A_1=1, A_{i+1}=6A_i-A_{i-1}$.
    In particular, we see that the torsion homology grows exponentially in $N$. 

	For the proof of~\cref{thm:lott_type_gap}, we will need two elementary linear algebra lemmas.
	
	\begin{lem} \label{lem:unif_transverse}
		Let $T$ be a linear transformation of a finite-dimensional real vector space $W$, and assume $W$ splits as a direct sum of eigenspaces for $T$ with eigenvalues $\neq \pm 1$. Let $W^{\pm} \subseteq W$ be subspaces whose dimensions sum to $\dim W$, such that $W^+$ (resp. $W^-$) has zero intersection with the contracting (resp. expanding) subspace of $W$ determined by $T$. Then whenever $k^{\pm}$ are nonnegative integers with $k^+ + k^-$ sufficiently large,
		\begin{align} \label{eqn:unif_decomp}
			W = T^{k^+} W^+ \oplus T^{-k^-} W^-,
		\end{align}
		and moreover this decomposition is uniformly transverse in $k^{\pm}$.
	\end{lem}

 By \textit{uniformly transverse}, we mean that there is some $\epsilon>0$ so that unit tangent vectors in the two subspaces make an angle of at least $\epsilon$.

	\begin{proof}
		Denote the contracting and expanding subspaces by $W^{\cont}$ and $W^{\expnd}$, respectively. Since $W^+ \cap W^{\cont} = 0$, we have $\dim W^+ \leq \dim W^{\expnd}$. Similarly, $\dim W^- \leq \dim W^{\cont}$. On the other hand, the dimensions of $W^{\pm}$ sum to $\dim W$, so these inequalities must be equalities. It follows from this and the zero intersection condition that $T^kW^+ \to W^{\expnd}$ and $T^{-k}W^- \to W^{\cont}$ as $k \to +\infty$. Thus \eqref{eqn:unif_decomp} holds, with uniform transversality, whenever both $k^+$ and $k^-$ are sufficiently large, say $k^{\pm} > K$.
		
		Now assume $k^+ + k^-$ is sufficiently large depending on $K$, so at least one of $k^{\pm}$ is large, say $k^+$ without loss of generality. Then $T^{k^+}W^+ \approx W^{\expnd}$, and the error in this approximation can be taken sufficiently small depending on $K$. It remains to check \eqref{eqn:unif_decomp} when $k^- = 1,\dots,K$. Since $W^- \cap W^{\expnd} = 0$, we have $T^{-k^-}W^- \cap W^{\expnd} = 0$ for all $k^-$. This combined with $T^{k^+}W^+ \approx W^{\expnd}$ gives \eqref{eqn:unif_decomp} for $k^- = 1,\dots,K$, with uniform transversality.
	\end{proof}

	\begin{lem} \label{lem:generic_growth}
		Let $T$ be a linear transformation of a finite-dimensional real vector space $W$, and assume $W$ splits as a direct sum of eigenspaces for $T$ with eigenvalues $\neq \pm 1$. Let $W^+ \subseteq W$ be a subspace which has zero intersection with the contracting subspace of $W$ determined by $T$. Then there is a constant $c < 1$, depending only on $T$, such that for all $w^+ \in W^+$ and $0\leq j \leq k$,
		\begin{align*}
			\|T^jw^+\| \lesssim c^{k-j} \|T^kw^+\|
		\end{align*}
		(here $\|\cdot\|$ is any fixed choice of norm on $W$). The implied constant depends only on $T$ and $W^+$ and $||\cdot||$; it is uniform in $w^+,j,k$.
	\end{lem}

	\begin{proof}
		All norms are equivalent, so we may as well assume $\|\cdot\|$ comes from an inner product which makes the eigenspaces of $T$ orthogonal. Since $W^+$ has zero intersection with the contracting subspace, so does $T^nW^+$ for all $n \geq 0$. Thus if $\Pi$ denotes the orthogonal projection onto the expanding subspace, then $\Pi|_{T^nW^+}$ is injective. Furthermore, the subspaces $T^nW^+$ converge to the expanding subspace, so $\Pi|_{T^nW^+}$ is ``uniformly injective": $\|\Pi w\| \sim \|w\|$ for all $w \in T^nW^+$, uniformly in $n$. Note also that $\Pi$ commutes with $T$. We can therefore write
		\begin{align*}
			\|T^jw^+\|
			\sim \|\Pi T^jw^+\|
			= \|T^j\Pi w^+\|
			\leq c^{k-j} \|T^k\Pi w^+\|
			= c^{k-j} \|\Pi T^kw^+\|
			\sim c^{k-j} \|T^kw^+\|,
		\end{align*}
		where $c = |\lambda|^{-1}$, and $\lambda$ is the eigenvalue which maximizes $|\lambda|$ among those $|\lambda| >1$.
	\end{proof}

	We can now prove Theorem \ref{thm:lott_type_gap}.

	\begin{proof}[Proof of Theorem \ref{thm:lott_type_gap}]
		Let $N$ be large, and let $\alpha$ be a closed (smooth) 2-form on $M_N$. By Lemma \ref{lem:dualchar}, it suffices to find a 1-form $\beta$ with $d\beta = \alpha$ and $\|\beta\|_{L^2} \lesssim \|\alpha\|_{L^2}$ (with constant independent of $N$). We will make a series of reductions so that in the end, we only need to manually construct $L^2$-bounded primitives $\beta$ for particularly simple 2-forms $\alpha$.
		
		
		\noindent
		\textit{Step 0.} The restriction map
		\begin{align*}
			H^2(M_N) \to H^2(\Sigma \times (-N-1,N+1))
		\end{align*}
		is the zero map. 
		
		This map can be factored as
		\begin{align*}
			H^2(M_N)
			\to H^2(E^+)
			\to H^2(\Sigma \times (N,N+1))
			\simeq H^2(\Sigma \times (-N-1,N+1)),
		\end{align*}
		and $H^2(E^+) = 0$ by (iii).
		
		\noindent
		\textit{Step 1.} Reduction to the case where
		\begin{align}\label{eqn:red1}
			\supp\alpha
			\subseteq E^+ \cup E^- \cup \bigcup_{|n| \leq N} \Sigma \times [n-0.1,n+0.1].
		\end{align}
	
		To begin with, $\alpha$ is an arbitrary closed $2$-form. For each $-N-1 \leq m \leq N$, let $\beta_m$ be a primitive of $\alpha|_{\Sigma \times (m+0.05,m+0.95)}$ with 
		\begin{align*}
			\|\beta_m\|_{L^2(\Sigma \times (m+0.05,m+0.95))} \lesssim \|\alpha\|_{L^2(\Sigma \times (m,m+1))}
		\end{align*}
		(such a $\beta_m$ exists by Step 0 and Proposition \ref{prop:non-cpt}, and the constant can be taken uniform in $m$ because of (ii)). Let $\chi_m$ be a smooth cutoff supported in $\Sigma \times (m+0.05,m+0.95)$ which is $1$ on $\Sigma \times (m+0.1,m+0.9)$. Choose the $\chi_m$'s to all be translates of each other, so that they obey uniform derivative bounds. 
  Then consider
		\begin{align*}
			\alpha' = \alpha - \sum_{-N-1 \leq m \leq N} d(\chi_m\beta_m).
		\end{align*}
		This is a closed 2-form satisfying the support condition \eqref{eqn:red1}. In addition,  $\|\alpha'\|_{L^2} \lesssim \|\alpha\|_{L^2}$, and
		\begin{align*}
			\Big\|\sum_{-N-1 \leq m \leq N} \chi_m \beta_m\Big\|_{L^2}
			\lesssim \|\alpha\|_{L^2}.
		\end{align*}
		Thus if we can find an $L^2$-bounded primitive for $\alpha'$, then we can find an $L^2$-bounded primitive for $\alpha$. This finishes Step 1.
		
		\noindent
		\textit{Step 2.} Reduction to the case where
		\begin{align} \label{eqn:red2}
			\supp\alpha
			\subseteq \bigcup_{|n| \leq N} \Sigma \times [n-0.1,n+0.1].
		\end{align}
	
		Let $\alpha$ be a closed 2-form for which \eqref{eqn:red1} holds. Let $\beta^{\pm}$ be a primitive of $\alpha|_{E^{\pm}}$ with $\|\beta^{\pm}\|_{L^2(E^{\pm})} \lesssim \|\alpha\|_{L^2}$, which exists by Proposition \ref{prop:non-cpt}. Let $\chi^{\pm}$ be a smooth cutoff which is $0$ on $E^{\mp} \cup \Sigma \times (-N,N)$, and which is $1$ outside of $E^{\mp} \cup \Sigma \times (-N-0.1,N+0.1)$. Consider
		\begin{align*}
			\alpha' = \alpha - d(\chi^+\beta^+) - d(\chi^-\beta^-).
		\end{align*}
		This is a closed 2-form satisfying \eqref{eqn:red2}, and we have the estimates $\|\alpha'\|_{L^2} \lesssim \|\alpha\|_{L^2}$ and $\|\chi^{\pm}\beta^{\pm}\|_{L^2} \lesssim \|\alpha\|_{L^2}$. Therefore finding an $L^2$-bounded primitive for $\alpha$ reduces to finding an $L^2$-bounded primitive for $\alpha'$. Step 2 is now complete.
		
		Denote
		\begin{align*}
			& V_N^+ = \Ker(V \simeq H_c^2(\Sigma \times (N,N+1)) \xrightarrow{(\iota_N^+)_*} H_c^2(E^+)),
			\\& V_N^- = \Ker(V \simeq H_c^2(\Sigma \times (-N-1,-N)) \xrightarrow{(\iota_N^-)_*} H_c^2(E^-)).
		\end{align*}
		Observe that
		\begin{align} \label{eqn:V_N_twist_of_V}
			V_N^{\pm}
			= \tilde{\varphi}_*^{\pm N} V^{\pm}.
		\end{align}
	
		\noindent
		\textit{Step 3.} Reduction to the case where
		\begin{align*}
			\alpha = d\beta + \sum_{|n| \leq N} (\alpha_n^+ + \alpha_n^-)
		\end{align*}
		with $\supp \alpha_n^{\pm} \subseteq \Sigma \times (n-0.1,n+0.1)$ and $[\alpha_n^{\pm}] \in V_N^{\pm}$, and also
		\begin{align} \label{eqn:Step_3_bds}
			\|\beta\|_{L^2}
			\lesssim \|\alpha\|_{L^2}
			\qquad \text{and} \qquad
			\|\alpha_n^{\pm}\|_{L^2}
			\lesssim \|\alpha\|_{L^2(\Sigma \times (n-0.1,n+0.1))}.
		\end{align}
		
		Let $\alpha$ be a closed 2-form for which \eqref{eqn:red2} holds. Write
		\begin{align*}
			\alpha = \sum_{|n| \leq N} \alpha_n
			\qquad \text{with} \qquad
			\supp\alpha_n \subseteq \Sigma \times [n-0.1,n+0.1].
		\end{align*}
		Fix a choice of a $2g$-dimensional subspace
		\begin{align*}
			W \subseteq \Omega_c^2(\Sigma \times (-0.1,0.1))^{\text{closed}}
		\end{align*}
		such that the natural map $W \to V$ is an isomorphism. Let $\omega \colon V \to W$ be the inverse map. So given $v \in V$ a cohomology class, $\omega(v)$ is a differential form representing $v$. Choose a norm $\|\cdot\|_V$ on $V$. By (ii) and finite dimensionality,
		\begin{align} \label{eqn:bd_by_omega}
			\|v\|_V \sim \|\omega(v)\|_{L^2}.
		\end{align}
		Also, by (ii) and Poincar\'e duality, for any closed 2-form $\eta$ supported in $\Sigma \times (-r,r)$ with $r = O(1)$, 
		\begin{align} \label{eqn:bd_class_by_form}
			\|[\eta]\|_V \lesssim \|\eta\|_{L^2}.
		\end{align}
		Since $N$ is large, it follows from \eqref{eqn:V_N_twist_of_V}, (iv), (v), and Lemma \ref{lem:unif_transverse} that for $|n| \leq N$,
		\begin{align*}
			V = \tilde{\varphi}_*^{N-n} V^+ \oplus \tilde{\varphi}_*^{-N-n} V^-
			= \tilde{\varphi}_*^{-n} V_N^+ \oplus \tilde{\varphi}_*^{-n} V_N^-,
		\end{align*}
		and furthermore, this decomposition is uniformly transverse for all $N,n$. Given $v \in V$, write $v = v^+ \oplus v^-$ for the decomposition of $v$ according to $V = V_N^+ \oplus V_N^-$. Then uniform transversality means that
		\begin{align} \label{eqn:unif_trans_N,n}
			\|\tilde{\varphi}_*^{-n} v\|_V
			\sim \|\tilde{\varphi}_*^{-n} v^+\|_V + \|\tilde{\varphi}_*^{-n} v^-\|_V
		\end{align}
		(uniformly in $N,n$).
		Denote $a_n = [\alpha_n] \in V$. Let
		\begin{align*}
			\alpha_n^{\pm}
			= \tilde{\varphi}_*^n \omega(\tilde{\varphi}_*^{-n} a_n^{\pm}).
		\end{align*}
		This is supported in $\Sigma \times (n-0.1,n+0.1)$.
		By (ii), \eqref{eqn:bd_by_omega}, \eqref{eqn:bd_class_by_form}, and \eqref{eqn:unif_trans_N,n},
		\begin{align*}
			\|\alpha_n^{\pm}\|_{L^2}
			\sim \|\tilde{\varphi}_*^{-n} a_n^{\pm}\|_V
			\lesssim \|\tilde{\varphi}_*^{-n} a_n\|_V
			\lesssim \|\tilde{\varphi}_*^{-n} \alpha_n\|_{L^2}
			\sim \|\alpha_n\|_{L^2}.
		\end{align*}
		By construction, $[\alpha_n] = [\alpha_n^+] + [\alpha_n^-]$, so by Proposition \ref{prop:cpt}, there is a primitive $\beta_n$ for $\alpha_n - \alpha_n^+ - \alpha_n^-$ supported in $\Sigma \times (n-0.2,n+0.2)$ with $\|\beta_n\|_{L^2} \lesssim \|\alpha_n\|_{L^2}$ (the implied constant can be taken uniform by (ii)). So
		\begin{align*}
			\alpha_n = \alpha_n^+ + \alpha_n^- + d\beta_n.
		\end{align*}
		Let
		\begin{align*}
			\beta = \sum_{|n| \leq N} \beta_n.
		\end{align*}
		Since the $\beta_n$ have disjoint support, $\|\beta\|_{L^2} \lesssim \|\alpha\|_{L^2}$. We conclude that
		\begin{align*}
			\alpha = d\beta + \sum_{|n| \leq N} (\alpha_n^+ + \alpha_n^-)
		\end{align*}
		is of the desired form, completing Step 3.
		
		From now on, let $\alpha$ be as in the statement of Step 3. Write
		\begin{align*}
			\alpha^{\pm}
			= \sum_{|n| \leq N} \alpha_n^{\pm}.
		\end{align*}
		By \eqref{eqn:Step_3_bds} and the fact that the $\alpha_n^{\pm}$ have disjoint support for different $n$, we have $\|\alpha^{\pm}\|_{L^2} \lesssim \|\alpha\|_{L^2}$. Thus it suffices to find an $L^2$-bounded primitive for $\alpha^{\pm}$. By symmetry, it's enough to show that we can always find an $L^2$-bounded primitive for $\alpha^+$. This follows from
		
		\noindent
		\textit{Step 4.} Each $\alpha_n^+$ has a primitive $\beta_n^+$ such that
		\begin{align*}
			\supp \beta_n^+ \subseteq \Sigma \times (n-1,N+1) \cup E^+
		\end{align*}
		and
		\begin{align}
			& \|\beta_n^+\|_{L^2(\Sigma \times (m-1,m+1))}
			\lesssim c^{m-n} \|\alpha_n^+\|_{L^2}
			\quad \text{(for all } n \leq m \leq N \text{)}, \label{eqn:Step_4_bd_1}
			\\& \|\beta_n^+\|_{L^2(E^+)}
			\lesssim c^{N-n} \|\alpha_n^+\|_{L^2},
                \label{eqn:Step_4_bd_2}
		\end{align}
		where $c<1$ is a constant independent of $N,n,m$. These estimates say that $\beta_n^+$ is essentially supported on $\Sigma \times (n-1,n+1)$, with an exponentially decaying tail on the right.
		
		Denote $a_n^+ = [\alpha_n^+] \in V_N^+.$ Let $\alpha_{n,n}^+ = \alpha_n^+$, and for $n < m \leq N$, let
		\begin{align*}
			\alpha_{n,m}^+
			= \tilde{\varphi}_*^m \omega(\tilde{\varphi}_*^{-m} a_n^+).
		\end{align*}
		This is supported in $\Sigma \times (m-0.1,m+0.1)$.
		Write
		\begin{align*}
			\alpha_n^+
			= \alpha_{n,N}^+ + \sum_{m=n}^{N-1} (\alpha_{n,m}^+ - \alpha_{n,m+1}^+).
		\end{align*}
		Since $[\alpha_{n,m}^+] = [\alpha_{n,m+1}^+]$ in $V$, Proposition \ref{prop:cpt} produces a $1$-form $\beta_{n,m}^+$ supported in $\Sigma \times (m-1,m+2)$ such that
		\begin{align*}
			d\beta_{n,m}^+
			= \alpha_{n,m}^+ - \alpha_{n,m+1}^+
			\qquad \text{and} \qquad
			\|\beta_{n,m}^+\|_{L^2}
			\lesssim \|\alpha_{n,m}^+\|_{L^2} + \|\alpha_{n,m+1}^+\|_{L^2}
		\end{align*}
		(as usual, the constant here can be taken to be uniform by (ii)). Also, $[\alpha_{n,N}^+] \in V_N^+$, so
		\begin{align*}
			\alpha_{n,N}^+
			\text{ vanishes in }
			H_c^2(\Sigma \times (N-1,N+1) \cup E^+).
		\end{align*}
		Thus by Proposition \ref{prop:cpt}, there is a 1-form $\beta_{n,N}^+$ supported in $\Sigma \times (N-1,N+1) \cup E^+$ such that
		\begin{align*}
			d\beta_{n,N}^+ = \alpha_{n,N}^+
			\qquad \text{and} \qquad
			\|\beta_{n,N}^+\|_{L^2}
			\lesssim \|\alpha_{n,N}^+\|_{L^2}
		\end{align*}
		(again, the constant can be taken uniform by (i),(ii)). Let
		\begin{align*}
			\beta_n^+
			= \sum_{m=n}^N \beta_{n,m}^+.
		\end{align*}
		Then $d\beta_n^+ = \alpha_n^+$, and \eqref{eqn:Step_4_bd_1} and \eqref{eqn:Step_4_bd_2} will hold if we can show that
		\begin{align*}
			\|\alpha_{n,m}^+\|_{L^2}
			\lesssim c^{m-n} \|\alpha_n^+\|_{L^2}
		\end{align*}
		for some constant $c < 1$. Estimate
		\begin{align*}
			\|\alpha_{n,m}^+\|_{L^2}
			&\sim \|\tilde{\varphi}_*^{-m} a_n^+\|_V
			\\&= \|\tilde{\varphi}_*^{-(m-n)} \tilde{\varphi}_*^{-n}a_n^+\|_V
			\\&\leq \|\tilde{\varphi}_*^{-(m-n)}\|_{\tilde{\varphi}_*^{-n} V_N^+ \to V} \|\tilde{\varphi}_*^{-n} a_n^+\|_V
			\\&\lesssim \|\tilde{\varphi}_*^{-(m-n)}\|_{\tilde{\varphi}_*^{N-n} V^+ \to V} \|\tilde{\varphi}_*^{-n}\alpha_n^+\|_{L^2}
			\\&\sim \|\tilde{\varphi}_*^{-(m-n)}\|_{\tilde{\varphi}_*^{N-n} V^+ \to V} \|\alpha_n^+\|_{L^2}.
		\end{align*}
		In the third line we use that $a_n^+ \in V_N^+$.
		It remains to prove the operator norm bound
		\begin{align*}
			\|\tilde{\varphi}_*^{-(m-n)}\|_{\tilde{\varphi}_*^{N-n} V^+ \to V} \lesssim c^{m-n}.
		\end{align*}
		Equivalently, we want to show that for all $v^+ \in V^+$,
		\begin{align*}
			\|\tilde{\varphi}_*^{N-m} v^+\|_V
			\lesssim c^{m-n} \|\tilde{\varphi}_*^{N-n}v^+\|_V.
		\end{align*}
		This is a direct consequence of (iv), (v), and Lemma \ref{lem:generic_growth}. Thus Step 4 is finished, and the theorem is proved.
	\end{proof}
    
\end{section} 
\section{Proof of Theorem \ref{thm:2}}\label{sec.pf:2}

\subsection{Geometric part}
\begin{prop}\label{prop:geometricpart}
Suppose a sequence of pointed closed hyperbolic 3-manifolds $\{(M_i,x_i)\}$ geometrically converges in the pointed Gromov--Hausdorff topology to a tame manifold $(M_\infty,x_\infty)$ with at least one end. Suppose further that there is a uniform lower bound $\lambda_0$ for the bottom eigenvalue of the Laplacian on coclosed 1-forms on $M_i$. Then there is a sequence of manifolds $M_N$ satisfying the hypotheses (i) and (ii) of \cref{thm:torsion_analytic} below, with bilipschitz constants uniform in $N$.
\end{prop}

\begin{proof}

There exists some $\epsilon=\epsilon(\lambda_0)<1$ such that the injectivity radius at each point of $M_i$ is between $\epsilon$ and $\frac{1}{\epsilon}$.  The limit manifold  $M_{\infty}$ thus inherits the same two-sided bounds. 
The upper uniform upper bound on the injectivity radius combined with tameness implies that the ends of $M_\infty$ are degenerate. Since the injectivity radius at each point is at least some uniform $\epsilon$, it follows from \cite{min94} that each end, parameterized as $\Sigma \times [0,\infty)$, is $K$-bilipschitz to a metric $g_t + dt^2$ where $g_t$ is a family of metrics on $\Sigma$ parameterized by a unit speed geodesic in Teichm\"uller space, for some $K$ depending only on $\eps$ and the genus of $\Sigma$. 

By the lower bound on the injectivity radius at any point, $g_t$ stays in the $\eps/{(10K)}$-thick part of Teichm\"uller space. Any two such metrics are $K'$-bilipschitz for some $K'$ depending only on $\eps$ and the genus of $\Sigma$. Therefore, $g_n$ satisfies conditions (i) and (ii) from the beginning of Section \ref{sec.pf:11} with bilipschitz constants depending only on $K$ and $K'$.
\end{proof}
\subsection{Analytic part}\label{sec.analyticpart}


Fix a large parameter $N \in \mathbb{Z}^+$ (all estimates will be uniform in $N$). Let $(M_N,g_N)$ be a smooth closed Riemannian 3-manifold, and assume given a diffeomorphic identification of an open subset of $M_N$ with $\Sigma \times (-N,N)$, where $\Sigma$ is a fixed surface of genus $g$. In addition, assume the following:
	\begin{enumerate} \parskip = 0.5em
		\item[(i)] For each $|n| < N$, there is a metric $g_n$ on $\Sigma$ such that $g_N|_{\Sigma \times (n-1,n+1)}$ is commensurate to the product metric $g_n + dt^2$ on $\Sigma \times (n-1,n+1)$, uniformly for all $|n| < N$.
		
		\item[(ii)] The surfaces $(\Sigma,g_n)$, with $g_n$ as in (i), are uniformly bilipschitz.
	\end{enumerate}

        Here when we say ``uniformly," we mean that the commensurability and bilipschitz constants should be thought of as being independent of $n,N$. We make this intuition rigorous by stipulating that in the remainder of this section, all implicit constants may depend on those in (i) and (ii), as well as on the genus $g$ and on the metric $g_0$ on $\Sigma$ in (i), but never on anything else --- in particular never on $n,N$. 

	The condition (ii) means that there are diffeomorphisms $\phi_n \colon \Sigma \to \Sigma$ such that the metrics $\phi_n^* g_n$ are uniformly commensurate.

	\begin{thm} \label{thm:torsion_analytic}
        Let $(M_N,g_N)$ as above satisfy (i) and (ii).
		Let $0 < \lambda_0 \lesssim 1$. Suppose the bottom eigenvalue of the Laplacian on coclosed 1-forms on $M_N$ is at least $\lambda_0$. Then there exists $\delta \gtrsim \lambda_0$ 
        such that
		\begin{align*}
			\#H_1(M_N,\mathbb{Z})_{\tors}
			\gtrsim \lambda_0^{2g}(1+\delta)^N.
		\end{align*}
	\end{thm}

    Combined with \cref{prop:geometricpart}, \cref{thm:torsion_analytic} will prove \cref{thm:2}.

	Since $0$ is never an eigenvalue, there are no nonzero harmonic $1$-forms on $M_N$, so $M_N$ is a rational homology sphere. Thus for each $|t| < N$, the surface $\Sigma \times \{t\}$ separates $M_N$ into two components, which we call $M_N^{>t}$ and $M_N^{<t}$.
	
	Denote $\Lambda = H_1(\Sigma,\mathbb{Z})$, and let
	\begin{align*}
		\Lambda_N^+
		= \Ker(\Lambda \to H_1(M_N^{>-1}, \mathbb{Z}))
		\qquad \text{and} \qquad
		\Lambda_N^-
		= \Ker(\Lambda \to H_1(M_N^{<1}, \mathbb{Z})),
	\end{align*}
	where the maps are induced by $\Sigma = \Sigma \times \{0\} \hookrightarrow M_N$. Then
	\begin{align*}
		\Ker(\Lambda \to H_1(M_N,\mathbb{Z}))
		= \Lambda_N^+ + \Lambda_N^-,
	\end{align*}
	because by subdivision, any homology class which dies in $M_N$ can be written as a sum of classes which die on the right and on the left, respectively. This means that
	\begin{align*}
		\frac{\Lambda}{\Lambda_N^+ + \Lambda_N^-} \hookrightarrow H_1(M_N,\mathbb{Z}) = H_1(M_N,\mathbb{Z})_{\tors},
	\end{align*}
	and hence
	\begin{align} \label{eqn:torsion_lower_bd}
		\#H_1(M_N,\mathbb{Z})_{\tors}
		\geq \#\Big(\frac{\Lambda}{\Lambda_N^+ + \Lambda_N^-}\Big).
	\end{align}
	Our strategy to prove the theorem is to show that any nonzero element of $\Lambda_N^+ + \Lambda_N^-$ is large.

	
	View $\Lambda \subseteq V$, where $V = \Lambda \otimes \mathbb{R} = H_1(\Sigma,\mathbb{R})$. Given $I \subseteq (-N,N)$ an open interval, $V \simeq H_c^2(\Sigma \times I, \mathbb{R})$. It will often be helpful to think of $V$ in this way, so that elements of $V$ are de Rham cohomology classes and can be represented by differential forms. Let
	\begin{align*}
		V_N^+
		= \Ker(V \to H_1(M_N^{>-1}, \mathbb{R}))
		\qquad \text{and} \qquad
		V_N^-
		= \Ker(V \to H_1(M_N^{<1}, \mathbb{R})),
	\end{align*}
	so $\Lambda_N^{\pm} \subseteq \Lambda \cap V_N^{\pm}$. Again, it will be helpful to think of $V_N^{\pm}$ as
	\begin{align*}
		V_N^+ = \Ker(V \to H_c^2(M_N^{>t}, \mathbb{R}))
		\qquad \text{and} \qquad
		V_N^- = \Ker(V \to H_c^2(M_N^{<t}, \mathbb{R}));
	\end{align*}
	this holds for any $|t| < N$. As above, by subdivision,
	\begin{align*}
		\Ker(V \to H_1(M_N,\mathbb{R})) = V_N^+ + V_N^-,
	\end{align*}
	but $M_N$ is a rational homology sphere, so $V_N^+ + V_N^- = V$. In fact, $V = V_N^+ \oplus V_N^-$, because $\dim V_N^{\pm} = \frac{1}{2} \dim V$ by half-lives-half-dies. Given $v \in V$, let
	\begin{align*}
		v = v_N^+ + v_N^-
	\end{align*}
	denote the corresponding decomposition.

	For each $|n| < N$, define a norm $\|\cdot\|_{\Sigma \times (n-1,n+1)}$ on $V$ by
	\begin{align*}
		\|v\|_{\Sigma \times (n-1,n+1)}
		= \inf_{\alpha} \|\alpha\|_{L^2(\Sigma \times (n-1,n+1))},
	\end{align*}
	where the infimum runs over all smooth closed $2$-forms $\alpha$ with compact support in $\Sigma \times (n-1,n+1)$ which represent the class $v \in H_c^2(\Sigma \times (n-1,n+1), \mathbb{R})$, and the $L^2$ norm is taken with respect to the product metric $g_n + dt^2$ from (i). If $v \neq 0$, then $\|v\|_{\Sigma \times (n-1,n+1)} \neq 0$ by Poincar\'e duality. By (i), this norm satisfies the property that if $\alpha$ is a smooth closed 2-form that is admissible in the above the infimum 
    , then
	\begin{align*}
		\|\alpha\|_{L^2(M_N)}
		\lesssim \|v\|_{\Sigma \times (n-1,n+1)}.
	\end{align*}
	Conversely, given $v \in V$, there exists such an $\alpha$ with
	\begin{align*}
		\|\alpha\|_{L^2(M_N)}
		\lesssim \|v\|_{\Sigma \times (n-1,n+1)}.
	\end{align*}
	
	Fix smooth functions $\chi^{\pm}$ on $\mathbb{R}$ such that $\chi^+ + \chi^- = 1$, and $\chi^+,\chi^-$ are supported in $(-1,\infty)$ and $(-\infty,1)$, respectively. For $|n| < N$, let $\chi_n^{\pm}$ be the translate
	\begin{align*}
		\chi_n^{\pm}(t) = \chi^{\pm}(t-n).
	\end{align*}
	View $\chi_n^{\pm}$ as a function on $\Sigma \times (-N,N)$ by pulling back, and then extend $\chi_n^{\pm}$ to a function on $M_N$ in the obvious way (so that $\chi_n^{\pm}$ is locally constant outside of $\Sigma \times (n-1,n+1)$). It follows from (i) and (ii) that
	\begin{align*}
		\|d\chi_n^{\pm}\|_{L^{\infty}(M_N)} \lesssim 1.
	\end{align*}

	\begin{lem}[Uniform transversality of $V_N^{\pm}$] \label{lem:transverse}
		For all $v \in V$ and $|n| < N$,
		\begin{align*}
			\|v_N^{\pm}\|_{\Sigma \times (n-1,n+1)}
			\lesssim \lambda_0^{-1/2} \|v\|_{\Sigma \times (n-1,n+1)}.
		\end{align*}
	\end{lem}

	Here $\lambda_0$ is as in the statement of Theorem \ref{thm:torsion_analytic}. This lemma is a quantitative form of the statement that $V = V_N^+ \oplus V_N^-$, so the proof just amounts to making the subdivision argument quantitative. We will only need the $n=0$ case of the lemma, but the proof is the same for all $n$.

	\begin{proof}
		Fix $v$ and $n$. Let $\alpha$ be a 2-form representing the class
        \begin{align*}
            v \in H_c^2(\Sigma \times (n-1,n+1), \mathbb{R}),
        \end{align*}
        such that
		\begin{align*}
			\|\alpha\|_{L^2(M_N)}
			\lesssim \|v\|_{\Sigma \times (n-1,n+1)}.
		\end{align*}
		By assumption, $M_N$ has a spectral gap $\geq \lambda_0$ on coclosed 1-forms, which is equivalent to a spectral gap $\geq \lambda_0$ on closed 2-forms (by duality via the Hodge star). Since $\alpha$ is closed, there is thus a closed 2-form $\omega$ on $M_N$ with $\Delta\omega = \alpha$ and $\|\omega\|_{L^2} \leq \lambda_0^{-1} \|\alpha\|_{L^2}$. Let $\beta = d^*\omega$. Then
		\begin{align*}
			d\beta
			= dd^*\omega
			= (dd^*+d^*d)\omega
			= \alpha
		\end{align*}
		because $\omega$ is closed, and
		\begin{align*}
			\|\beta\|_{L^2}^2
			= \langle d^*\omega, d^*\omega \rangle_{L^2}
			= \langle dd^*\omega, \omega \rangle_{L^2}
			= \langle \alpha,\omega \rangle_{L^2}
			\leq \|\alpha\|_{L^2} \|\omega\|_{L^2}
			\leq \lambda_0^{-1} \|\alpha\|_{L^2}^2.
		\end{align*}
		Let $\alpha^{\pm} = d(\chi_n^{\pm} \beta)$. Since $\chi_n^+\beta$ is supported in $M_N^{>n-1}$, we have $[\alpha^+] \in V_N^+$. Similarly, $[\alpha^-] \in V_N^-$. In addition, $\alpha = \alpha^+ + \alpha^-$, so we must have $[\alpha^{\pm}] = v_N^{\pm}$. Finally, $\alpha^{\pm}$ is supported in $\Sigma \times (n-1,n+1)$, so
		\begin{align*}
			\|v_N^{\pm}\|_{\Sigma \times (n-1,n+1)}
			&\lesssim \|\alpha^{\pm}\|_{L^2(M_N)}
			\leq \|\chi_n^{\pm}\|_{L^{\infty}} \|\alpha\|_{L^2} + \|d\chi_n^{\pm}\|_{L^{\infty}} \|\beta\|_{L^2}
			\\&\lesssim \lambda_0^{-1/2} \|\alpha\|_{L^2}
			\lesssim \lambda_0^{-1/2} \|v\|_{\Sigma \times (n-1,n+1)},
		\end{align*}
		as desired.
	\end{proof}

	The next two lemmas identify contrasting behavior.
	
	\begin{lem}[Uniform lower bound for integral classes] \label{lem:integrality_bd}
		For all nonzero $v \in \Lambda$ and all $|n| < N$,
		\begin{align*}
			\|v\|_{\Sigma \times (n-1,n+1)}
			\gtrsim 1.
		\end{align*}
	\end{lem}
	
	\begin{lem}[Exponential decay for classes in $V_N^{\pm}$] \label{lem:V+-_exp_decay}
		There exists $\delta \gtrsim \lambda_0$ such that for all $v_N^{\pm} \in V_N^{\pm}$,
		\begin{align*}
			\|v_N^+\|_{\Sigma \times (N-2,N)}
			\lesssim \lambda_0^{-1/2} (1-\delta)^N \|v_N^+\|_{\Sigma \times (-1,1)}
		\end{align*}
        and
        \begin{align*}
            \|v_N^-\|_{\Sigma \times (-N,-N+2)}
			\lesssim \lambda_0^{-1/2} (1-\delta)^N \|v_N^-\|_{\Sigma \times (-1,1)}.
        \end{align*}
	\end{lem}

	Before proving these lemmas, let us see how we can exploit the tension between them to deduce Theorem~\ref{thm:torsion_analytic}.
	
	\begin{proof}[Proof of Theorem \ref{thm:torsion_analytic} assuming Lemmas \ref{lem:integrality_bd} and \ref{lem:V+-_exp_decay}]
		Let $v \in \Lambda_N^+ + \Lambda_N^-$ be nonzero. In view of \eqref{eqn:torsion_lower_bd}, we wish to show that $v$ is large. Since $v \neq 0$, one of $v_N^{\pm}$ is nonzero; assume $v_N^+ \neq 0$ without loss of generality. Then for some $\delta \gtrsim \lambda_0$,
		\begin{align*}
			1 \lesssim \|v_N^+\|_{\Sigma \times (N-2,N)}
			\lesssim \lambda_0^{-1/2} (1-\delta)^N \|v_N^+\|_{\Sigma \times (-1,1)}
			\lesssim \lambda_0^{-1} (1-\delta)^N \|v\|_{\Sigma \times (-1,1)},
		\end{align*}
		where these three inequalities, in order, are by Lemma \ref{lem:integrality_bd}, Lemma \ref{lem:V+-_exp_decay}, and Lemma \ref{lem:transverse}. Consequently
		\begin{align*}
			\|v\|_{\Sigma \times (-1,1)}
			\gtrsim \lambda_0 (1+\delta)^N.
		\end{align*}
		Fix a norm $\|\cdot\|$ on $V$ which makes $\Lambda \subseteq V$ isometric to the standard lattice $\mathbb{Z}^{2g} \subseteq \mathbb{R}^{2g}$. Since $\|\cdot\|_{\Sigma \times (-1,1)}$ depends only on the metric $g_0$ on $\Sigma$ in (i), and we allow our implied constants to depend both on $g_0$ and on the genus $g$, we have $\|\cdot\|_{\Sigma \times (-1,1)} \sim \|\cdot\|$. Therefore
		\begin{align*}
			\|v\| \gtrsim \lambda_0(1+\delta)^N.
		\end{align*}
		It follows that there exists $R \gtrsim \lambda_0 (1+\delta)^N$, such that if $B_R$ denotes the ball of radius $R$ around the origin in $(V,\|\cdot\|)$, then
		\begin{align} \label{eqn:inject_ball}
			\Lambda \cap B_R
			\hookrightarrow \frac{\Lambda}{\Lambda_N^+ + \Lambda_N^-}.
		\end{align}
		The set $\Lambda \cap B_R$ has size
		\begin{align*}
			\#(\Lambda \cap B_R)
			\gtrsim R^{2g}
			\gtrsim \lambda_0^{2g} (1+\delta)^{2gN}.
		\end{align*}
		Equivalently (up to multiplying $\delta$ by a constant $\sim 1$),
		\begin{align} \label{eqn:ball_size}
			\#(\Lambda \cap B_R) \gtrsim \lambda_0^{2g} (1+\delta)^N.
		\end{align}
		Combining \eqref{eqn:torsion_lower_bd}, \eqref{eqn:inject_ball}, and \eqref{eqn:ball_size} yields the result.
	\end{proof}

	It remains to prove Lemmas \ref{lem:integrality_bd} and \ref{lem:V+-_exp_decay}.
	
	\begin{proof}[Proof of Lemma \ref{lem:integrality_bd}]
		By (ii), there are bilipschitz homeomorphisms $\phi_n \colon \Sigma \to \Sigma$ such that the metrics $\phi_n^*g_n$ are uniformly commensurate. Composing all the $\phi_n$ with $\phi_0^{-1}$, we may assume $\phi_0 = \id$. Then
		\begin{align*}
			\|v\|_{\Sigma \times (n-1,n+1)}
			\sim \|\phi_n^*v\|_{\Sigma \times (-1,1)}.
		\end{align*}
		Since $v$ is nonzero and integral, so is $\phi_n^*v$. Therefore
		\begin{align*}
			\|v\|_{\Sigma \times (n-1,n+1)}
			\gtrsim \min_{0 \neq w \in \Lambda} \|w\|_{\Sigma \times (-1,1)}. 
		\end{align*}
		The right hand side is a positive number depending only on $g_0$, so it is $\gtrsim 1$ because we allow implied constants to depend on $g_0$.
	\end{proof}

	\begin{proof}[Proof of Lemma \ref{lem:V+-_exp_decay}]
		The argument is similar to the proof of Lemma \ref{lem:transverse}, but a little more complicated. We only prove the desired bound for $v_N^+$, as the bound for $v_N^-$ is completely symmetric.
		
		Fix $0 \leq m < N-1$, and let $\alpha$ be a 2-form representing the class
        \begin{align*}
            v_N^+ \in H_c^2(\Sigma \times (m-1,m+1), \mathbb{R}),
        \end{align*}
        such that
		\begin{align*}
			\|\alpha\|_{L^2(M_N)}
			\lesssim \|v_N^+\|_{\Sigma \times (m-1,m+1)}.
		\end{align*}
		By the same method as in the proof of Lemma \ref{lem:transverse}, we can find a primitive $\beta$ of $\alpha$ with $\|\beta\|_{L^2} \leq \lambda_0^{-1/2} \|\alpha\|_{L^2}$. For $m < n < N$, let $\alpha_n^+ = d(\chi_n^+\beta)$. Then again by the same reasoning as in Lemma \ref{lem:transverse}, we have $[\alpha_n^+] = v_N^+$. The construction of $\chi_n^+$ depended on a choice of smooth function $\chi^+$ on $\mathbb{R}$; we may assume $\chi^+$ is supported in $(\frac{1}{2},\infty)$. Then since $n>m$, the supports of $\chi_n^+$ and $\alpha$ are disjoint, so $\alpha_n^+ = d\chi_n^+ \wedge \beta$ is supported in $\Sigma \times (n-1,n+1)$, and
		\begin{align*}
			\|v_N^+\|_{\Sigma \times (n-1,n+1)}
			\lesssim \|\alpha_n^+\|_{L^2}
			\lesssim \|\beta\|_{L^2(\Sigma \times (n-1,n+1))}.
		\end{align*}
		Squaring and summing in $n$,
		\begin{align*}
			\sum_{n=m+1}^{N-1} \|v_N^+\|_{\Sigma \times (n-1,n+1)}^2
			\lesssim \|\beta\|_{L^2(M_N)}^2
			\leq \lambda_0^{-1} \|\alpha\|_{L^2}^2
			\lesssim \lambda_0^{-1} \|v_N^+\|_{\Sigma \times (m-1,m+1)}^2.
		\end{align*}
		The desired estimate now follows from the elementary lemma below.
	\end{proof}

	\begin{lem}
		Let $C > 0$, and let $a_0,\dots,a_N$ be a sequence of positive real numbers such that
		\begin{align*}
			\sum_{n=m+1}^{N} a_n \leq Ca_m
		\end{align*}
		for all $m < N$. Then
		\begin{align*}
			a_N \leq \frac{C}{(1+C^{-1})^{N-1}} a_0.
		\end{align*}
	\end{lem}

	\begin{proof}
		Induction.
	\end{proof}

\subsection{Proof of Corollary~\ref{cor:21}}
Recall that the \emph{rank} of a finitely generated group is the minimal number of generators in a generating set. The following proposition is proven in \cite{bs11}.
\begin{prop}[Proposition 6.1 \cite{bs11}]\label{prop:bslemma}
Assume that $\{M_i\}$ is a sequence of pairwise distinct hyperbolic 3-manifolds with $\inj(M_i) \geq \varepsilon$ and $\rank(\pi_1(M_i)) \leq k$. Then there are points $x_i \in M_i$ such that, up to passing to a subsequence, the pointed manifolds $(M_i, x_i)$ converge in the pointed Gromov--Hausdorff topology to a pointed hyperbolic 3-manifold $(M_\infty,x_\infty)$ homeomorphic to $\Sigma \times \R$ that has two degenerate ends. Here, $\Sigma$ is a closed, orientable surface with genus at most $k$.
\end{prop}

\begin{proof}[Proof of \cref{cor:21}]
Suppose there are infinitely many pairwise distinct hyperbolic manifolds $\{M_i\}$ with a uniform lower bound for the spectral gap for coexact 1-forms and a uniform upper bound for $\rank(\pi_1(M_i))$. The uniform lower bound for the spectral gap implies a uniform lower bound for the injectivity radius. \cref{prop:bslemma} then applies and yields a pointed subsequence $\{M_i, x_i\}$ converging in the pointed Gromov--Hausdorff topology to a pointed hyperbolic 3-manifold $(M_\infty,x_\infty)$ homeomorphic to $\Sigma \times \R$. This doubly degenerate manifold also has a uniform lower bound on its injectivity radius.  For any $N>0$, let $i_N$ be large enough that the radius $N$ ball around $x_{i_N}$ in $M_{i_N}$ is at Gromov--Hausdorff distance $\ll \varepsilon$ from the corresponding ball in $M_{\infty}$. Then the subsequence $M_{i_N}$ satisfies the hypotheses of \cref{thm:torsion_analytic}. We conclude that $|H_1(M_{i_N})|$ grows without bound.
\end{proof}

 \subsection{Proof of Corollary~\ref{cor:22}} We now explain why \cref{thm:torsion_analytic} implies that the torsion homology of a sequence of hyperbolic rational homology 3-spheres $M_n$ with a coexact 1-form spectral gap grows exponentially in volume, assuming that there is a uniform bound on the rank, or number of generators, of $\pi_1(M_n)$.

A consequence of the main theorem of \cite[Theorem 14.4]{bs23} can be  stated as follows. 
 A closed hyperbolic 3-manifold of injectivity radius at least $\epsilon$ and rank at most $k$ (we denote the collection of all such manifolds by $\mathcal{M}(k,\epsilon)$) can be decomposed as a union of building block manifolds glued together along product regions homeomorphic to $\Sigma \times [0,1]$. The number of terms in this decomposition is at most $n(k)$, and each building block manifold has diameter and volume at most $B(k,\epsilon)$ and is homeomorphic to a member of a finite collection of compact manifolds with boundary depending only on $k$ and $\epsilon$. Moreover, they produce uniformly bilipschitz models for every element of $\mathcal{M}(k,\epsilon)$. The bilipschitz models here are exactly as in Section \ref{sec.gluing}.)
 
 It follows that there are a uniformly bounded number of product regions in any element of $\mathcal{M}(k,\epsilon)$, and, using the uniformly bilipschitz models, that for some $\delta=\delta(k,\epsilon)>0$ one of the product regions has volume at least $\delta$ times the volume of the ambient space.  Since the volume of this product region is proportional to the $N$ in the proof of Theorem \ref{thm:2}, we can run the same argument to show that the torsion homology grows exponentially in volume.

 \begin{section}{Questions and remarks}\label{sec.q}
    \begin{ques}\label{q.sar}\footnote{asked by Peter Sarnak}
Does there exist a sequence of hyperbolic rational homology spheres with a uniform spectral gap for both the Laplacian on (mean zero) functions and for the Laplacian on coexact 1-forms? 
 \end{ques}
Equivalently, the question asks if one can find an infinite family of high dimensional spectral expanders that are hyperbolic 3-manifolds. The question is already interesting when the condition of hyperbolicity is relaxed. The last author gives an example of a sequence of Riemannian 3-manifolds with bounded geometry having uniform spectral gaps for functions and 1-forms in \cite{jz}. The paper \cite{bs11} suggests that for a sequence of rational homology spheres to have a function spectral gap, the rank of the manifolds in the sequence must be going to infinity.    
     \begin{ques}
         Does there exist a sequence of \underline{arithmetic} hyperbolic rational homology 3-spheres with volume going to infinity and uniform spectral gap on coexact 1-forms?
     \end{ques}
    The study of the spectrum of the Laplacian on congruence manifolds has a long history beginning with the work of A. Selberg.
	Bergeron--Clozel \cite{bergeron_clozel} in particular studied the question of whether a uniform lower bound for the spectral gap of the Laplacian acting on $i$-forms exists for congruence subgroups of a connected algebraic group via representation theoretic methods giving information about the automorphic spectrum.
	In our case, since $SO(3, 1)$ does not have a discrete series, such a uniform lower bound for congruence manifolds (or for a sequence of congruence manifolds) cannot exist and the construction of a sequence of arithmetic manifolds satisfying the above properties requires considering a different class of arithmetic manifolds.
 
     One reason to expect the existence of such a sequence is the analogy between the spectra of coexact 1-forms on hyperbolic 3-manifolds and the spinor spectra of hyperbolic surfaces where such a sequence of arithmetic surfaces is known to exist by forthcoming work of the second and third authors. Such a sequence must have infinitely many 3-manifolds in the same commensurability class as otherwise the sequence will Benjamini--Schramm converge to $\mathbb{H}^3$ due to the results of Fraczyk and Raimbault (see Theorem A of \cite{fr19}). 

    \begin{ques}
        If a sequence of hyperbolic rational homology spheres has uniform spectral gap for coexact 1-forms and Benjamini--Schramm converges to a tame manifold (or a distribution on tame manifolds), must $|H_1(M_i)_{\textnormal{tors}}|$ grow exponentially in $\textnormal{Vol}(M_i)$?
    \end{ques}
    This is a strengthening of \cref{cor:22}. Although the Benjamini--Schramm limit is tame, the manifolds in the sequence may have fundamental groups of unbounded rank. (See \cite{bsdefinition} for a definition of Benjamini--Schramm convergence.)  

\begin{ques}
    Recall the notation from \cref{sec:bassnote}. For $i=1,2,3$, does $\textnormal{Bass}_{\mathcal{F}(\mathcal{M}_i)}$ contain an interval? What is the largest point in $\textnormal{Bass}_{\mathcal{F}(\mathcal{M}_i)}$? What about the largest limit point?

\end{ques}
There are already a few facts about the bass note spectra that we can glean from our arguments. Indeed, in all our constructions, one can add a large Dehn surgery to one of the handlebody building blocks while keeping it hyperbolic and a rational or integer homology sphere. This guarantees that our manifolds have an arbitrarily small upper bound on their injectivity radius. This implies, by Proposition 10 of~\cite{lott} or \cite{mph90}, that the 1-form spectral gap is bounded above by an arbitrarily small constant $c$. Thus we see that the non-rigid bass note spectra $\textnormal{Bass}_{\mathcal{F}(\mathcal{M}_1)}$, $\textnormal{Bass}_{\mathcal{F}(\mathcal{M}_2)}$, and $\textnormal{Bass}_{\mathcal{F}(\mathcal{M}_3)}$ must contain \emph{infinitely} many limit points.

 \end{section}

\printbibliography
\vspace{1cm}
\end{document}